\documentclass[pdflatex,sn-mathphys-num]{sn-jnl}

\usepackage{graphicx}%
\usepackage{multirow}%
\usepackage{amsmath,amssymb,amsfonts}%
\usepackage{amsthm}%
\usepackage{mathrsfs}%
\usepackage[title]{appendix}%
\usepackage{xcolor}%
\usepackage{textcomp}%
\usepackage{manyfoot}%
\usepackage{booktabs}%
\usepackage{algorithm}%
\usepackage{algpseudocode}%
\usepackage{listings}%
\usepackage{graphics} 
\usepackage{mathtools}
\usepackage{amsfonts}
\usepackage{epstopdf}
\usepackage{mathdots}
\usepackage[shortlabels]{enumitem}
\usepackage{caption}
\usepackage{subcaption}
\usepackage{amsopn}
\usepackage{array}
\usepackage{tabularx}

\newtheorem{theorem}{Theorem}
\newtheorem{lemma}{Lemma}
\newtheorem{hypothesis}{Hypothesis}
\newtheorem{corollary}{Corollary}

\newtheorem{proposition}[theorem]{Proposition}%

\theoremstyle{thmstyletwo}%
\newtheorem{example}{Example}%
\newtheorem{remark}{Remark}%

\theoremstyle{thmstylethree}%
\newtheorem{definition}{Definition}%

\DeclareMathOperator*{\rank}{rank}

\DeclarePairedDelimiter\abs{\lvert}{\rvert}

\def\R{\mathbb{R}}
\def\C{\mathbb{C}}
\newcommand{\mi}[1]{\mathbf{#1}}
\newcommand{\vet}[1]{\mathbf{#1}}

\newcommand{\cred}[1]{\textcolor{black}{#1}}

\begin{document}
\title{Symbol-based multilevel block $\tau$ preconditioners for multilevel block Toeplitz systems: GLT-based analysis and applications}


\author*[1]{\fnm{Sean Y.} \sur{Hon}}\email{seanyshon@hkbu.edu.hk}

\author[1]{\fnm{Congcong} \sur{Li}}\email{22482245@life.hkbu.edu.hk}

\author[2]{\fnm{Rosita L.} \sur{Sormani}}\email{rl.sormani@uninsubria.it}

\author[3]{\fnm{Rolf} \sur{Krause}}\email{rolf.krause@usi.ch}

\author[4]{\fnm{Stefano} \sur{Serra-Capizzano}}\email{s.serracapizzano@uninsubria.it}

\affil*[1]{\orgdiv{Department of Mathematics}, \orgname{Hong Kong Baptist University}, \orgaddress{\street{Kowloon Tong}, \city{Hong Kong SAR}}}

\affil[2]{\orgdiv{Department of Theoretical and Applied Sciences}, \orgname{University of Insubria}, \orgaddress{ \country{Italy}}}

\affil[3]{\orgdiv{Center for Computational Medicine in Cardiology}, \orgname{University of Italian Switzerland}, \orgaddress{\country{ Switzerland}}}

\affil[4]{\orgdiv{Department of Science and High Technology}, \orgname{University of Insubria}, \orgaddress{\country{ Italy}}}



\abstract{
In recent years, there has been a renewed interest in preconditioning for multilevel Toeplitz systems, a research field that has been extensively explored over the past several decades. This work introduces novel preconditioning strategies using multilevel $\tau$ matrices for both symmetric and nonsymmetric multilevel Toeplitz systems. Our proposals constitute a general framework, as they are constructed solely based on the generating function of the multilevel Toeplitz coefficient matrix, when it can be defined. We begin with nonsymmetric systems, where we employ a symmetrization technique by permuting the coefficient matrix to produce a real symmetric multilevel Hankel structure. We propose a multilevel $\tau$ preconditioner tailored to the symmetrized system and prove that the eigenvalues of the preconditioned matrix sequence cluster at $\pm 1$, leading to rapid convergence when using the preconditioned minimal residual method. The high effectiveness of this approach is demonstrated through its application in solving space fractional diffusion equations. Next, for symmetric systems we introduce another multilevel $\tau$ preconditioner and show that the preconditioned conjugate gradient method can achieve an optimal convergence rate, namely a rate that is independent of the matrix size, when employed for a class of ill-conditioned multilevel Toeplitz systems. Numerical examples are provided to critically assess the effectiveness of our proposed preconditioners compared to several leading existing preconditioned solvers, highlighting their superior performance.
}



\maketitle

\section{Introduction}
Preconditioning for Toeplitz systems has been a long-standing research focus over the past few decades. A Toeplitz matrix $T_n(f)$ is defined from the Fourier coefficients of a Lebesgue integrable function $f$ and the properties of $f$ and those of $T_n(f)$ are strictly related. In particular, when $f$ is nonnegative almost everywhere (a.e.) and not identically zero, the matrix $T_n(f)$ is Hermitian positive definite for any $n$ with order of ill-conditioning depending on the order of the zeros of $f$, if they exist \cite{S-BIT94,extr1,extr2,BoGr}; see also \cite{serra_block1,serra_block2} for the block case i.e. for the case where the generating function is Hermitian nonnegative definite matrix-valued. If in addition $f$ is even a.e., then $T_n(f)$ is real symmetric and positive definite. For Hermitian positive definite Toeplitz matrices $T_n(f)$, the efficiency of preconditioned iterative solvers has been comprehensively studied. Various highly efficient preconditioners for $T_n(f)$ have been proposed, including circulant \cite{Strang1986, doi:10.1137/0909051} and $\tau$ \cite{Bini1990} preconditioners. For more detailed information, refer to the comprehensive expositions in \cite{MR2376196, MR2108963, Chan:1996:CGM:240441.240445, GaroniCapizzano_one, GaroniCapizzano_two} and the references therein, regarding preconditioning, multigrid, and spectral analysis, also in the block setting.

Much less work has been devoted to developing preconditioners in cases where the Toeplitz matrix $T_n(f)$ is real nonsymmetric, because the nonsymmetric character prevents a descriptive analysis of the convergence rate. A natural choice of preconditioned Krylov subspace solvers in this setting is the generalized minimal residual (GMRES) method. However, as extensively discussed in \cite{ANU:9672992}, a significant negative property of GMRES is that its preconditioning is largely based on heuristics and, in general, its convergence behaviours cannot be rigorously analyzed solely through the knowledge of the eigenvalues \cite{Greenbaum_1996}.

An interesting solution to this problem stemmed from \cite{doi:10.1137/140974213}, in which a symmetrization strategy that can circumvent the theoretical difficulty of GMRES was developed. More precisely, a nonsymmetric Toeplitz matrix $T_n(f)$ can be symmetrized by premultiplying it with the exchange matrix $Y_n\in \mathbb{R}^{n \times n}$ defined as
\[
Y_{n}=\begin{bmatrix}{}
& & 1 \\
& \iddots & \\
1 &  & \end{bmatrix},
\]
thus resulting in a symmetric matrix $Y_n T_n(f)$ (to be exact, a Hankel matrix) and allowing the use of the minimal residual (MINRES) method. Since the convergence behaviour of this method is related only to the eigenvalues, effective preconditioners can now be constructed by exploiting the available spectral information.

For the symmetrized Toeplitz system thus obtained, various preconditioning techniques have been subsequently proposed. The authors in \cite{doi:10.1137/140974213} considered absolute value circulant preconditioners, demonstrating that the spectra of the preconditioned matrix sequence \cred{cluster} at $\pm 1$. In this respect, using a different approach, we observe that the work \cite{serra_indef} is the first where such kind of cluster at $\pm 1$ is studied in the context of Toeplitz preconditioning. This preconditioner could potentially lead to fast MINRES convergence, but we will show in the numerical experiments that it is not effective in general. As studied in \cite{Hon_SC_Wathen}, circulant preconditioners cannot guarantee rapid convergence, especially in the ill-conditioned Toeplitz case since the preconditioned matrix may have eigenvalue outliers very close to zero, and it is well-known that matrix algebra preconditioners with unitary transforms cannot lead to optimal convergence in general in the multilevel Toeplitz case \cite{SeTy-nega1,SeTy-nega2,S-nega-gen,nega-spectr-eq}. As an alternative, following \cite{doi:10.1137/140974213}, Toeplitz preconditioners and band-times-circulant preconditioners were introduced in \cite{Pestana2019} and \cite{Hon_SC_Wathen,NT_BIT,NT_JCAM}, respectively. In particular, a band-Toeplitz preconditioner was then proposed in \cite{Hon_SC_Wathen} for symmetrized Toeplitz matrices whose generating functions have zeros of even order. In our case, the generating function has zeros of noninteger order, and hence this strategy is not applicable.

Recently, Pestana proposed two ideal Toeplitz preconditioners \cite{Pestana2019} for the linear system arising from the Riemann-Liouville (R.-L.) fractional problem considered here. Although these ideal preconditioners can lead to rapid MINRES convergence, their implementation is computationally expensive, which hinders their practicality as general preconditioners. Subsequently, an efficient multilevel $\tau$ preconditioner was developed in \cite{Hon_2024} for the same problem. Numerical examples therein demonstrated that such a preconditioner could achieve optimal MINRES convergence, but it remains an open question as to whether or not the preconditioning effect of this preconditioner can be generalized. Moreover, as will be shown in our numerical tests, this preconditioner fails to be efficient when the fractional order approaches one.

Starting from the theory of multilevel block generalized locally Toeplitz (GLT) matrix sequences \cite{GaroniCapizzano_one,GaroniCapizzano_two,GaroniCapizzano_three,GaroniCapizzano_four}, the main contribution of this work consists in general tools for the spectral analysis of nonpreconditioned and preconditioned matrix sequences and in the further development, in the multilevel setting, of the $\tau$ preconditioning approach introduced in \cite{Hon_2024}. Specifically, we demonstrate that an efficient multilevel $\tau$ preconditioner can be constructed based only on the complex-valued generating function of the underlying nonsymmetric multilevel Toeplitz matrix. The strategy can be extended to the block case by using a block $\tau$ algebra and in that setting we require that the generating function is matrix-valued and that the blocks are real and symmetric (such a requirement can be weakened to just Hermitian if the block circulant algebra is used). Furthermore, we show that the eigenvalues of the preconditioned matrix sequence are clustered at $\pm 1$ under suitable assumptions, which supports fast convergence when MINRES is employed. To illustrate the applicability of our proposed preconditioning strategy, we employ it for \cred{R.-L.} fractional diffusion equations.

Another goal of this paper is the development of efficient $\tau$ preconditioners for a class of ill-conditioned symmetric multilevel Toeplitz systems. Thanks to the symmetric character, we are able to prove that the eigenvalues of the preconditioned coefficient matrix remain bounded between two positive constants, 
so that the condition number is bounded from above and the convergence rate of the preconditioned conjugate gradient (PCG) method is optimal, i.e., it is independent of the matrix size.

The paper is organized as follows. In Section \ref{sec:prelim}, we review some preliminary definitions and results. In Section \ref{sec:properties-symmetrized}, we present the symmetrization strategy for nonsymmetric Toeplitz systems and develop the theory regarding their asymptotic spectral distribution, by exploiting Toeplitz and GLT technical tools. In Section \ref{sec:spectral_nonsymmetric}, we apply these results to the spectral analysis of a discretized fractional problem and construct our proposed $\tau$ preconditioners, proving the cluster at $\pm 1$ of the preconditioned matrix sequence. In Section \ref{sec:symmetric}, we proceed to develop a further $\tau$ preconditioner in the symmetric setting. Numerical examples are provided in Section \ref{sec:numerical} to demonstrate the performance of our preconditioners, to compare them with existing solvers as referenced in \cite{Pestana2019, Hon_2024, huang2021spectral}, and to support the theoretical results.

\section{Preliminaries on spectral tools}
\label{sec:prelim}
We start with a summary of the background knowledge necessary to understand the main contents of the paper. The reader who wishes to learn more can consult \cite{GaroniCapizzano_one,GaroniCapizzano_two,GaroniCapizzano_three,GaroniCapizzano_four}, from which we adopt many notations and terminology, and references therein.

\subsection{Notation}
In what follows, we extensively use multi-indices. The notation $\mathbf{n}$ represents a multi-index $\mi{n} = (n_1, n_2, \dots, n_k)$, where $k$ is fixed and each $n_j$ is a positive integer. The product of all its components is denoted by $N(\mathbf{n})=n_1n_2\ldots n_k$. The expression ${{\bf n}\to\infty}$ means that every component of the vector $\mathbf{n}$ tends to infinity, i.e., $\min_{1\le j\le k} n_j\to \infty$. Operations and relations are defined componentwise; for example, $\mathbf{n}\ge\mathbf{m}$ means that the inequality holds for every component of $\mathbf{n}$ and $\mathbf{m}$, while $\mathbf{0,1,2,\ldots}$ represent vectors of all zeros, ones, twos, and so on.

We also use the Kronecker tensor product for matrices, denoted with $\otimes$. We recall that the Kronecker product between matrices $A\otimes B$ is the block matrix of the form $\left(a_{i,j} B\right)$ with $A= \left(a_{i,j}\right)$.

Regarding the employed norms, we adopt the following notations: $\|\cdot\|$ is the spectral norm for matrices, i.e. the maximal singular value, $\|\cdot\|_{\text{tr}}$ is the trace norm for matrices, computed as the sum of all the singular values, and $\|\cdot\|_{L^1}$ is the standard $L^1$ norm for functions.

We denote by ${C}_c(\R)$ (resp., ${ C}_c(\C)$) the space of continuous complex-valued functions with bounded support defined on $\R$ (resp., $\C$). Whenever measurability is invoked, we always refer to the Lebesgue measure $\mu_k$ on $\R^k$, unless otherwise specified. Finally, we recall that the {\em essential range} of a measurable function $\psi:D\subseteq\R^t\to\C$ is its non negligible range, i.e., $\mathrm{ess\,im}\psi := \big\{z\in\C :\, \forall\epsilon>0 \;\; \mu_t \{ \psi\in B(z,\epsilon)\} > 0 \big\}$. If $\psi$ is real valued, the {\em essential infimum} is the infimum of the essential range, i.e., $\mathrm{ess\,inf}\psi := \inf \big(\mathrm{ess\,im}\psi\big)$, and the {\em essential supremum} $\mathrm{ess\,sup}\psi$ is defined in the same way.

\subsection{Distribution and clustering}
A central notion is that of distribution of a matrix sequence, both in the sense of eigenvalues and of singular values. By matrix sequence, we mean any sequence composed of square matrices of strictly increasing size, so that it tends to infinity. The most general case considered in this work involves matrix sequences of the form $\{A_{\mathbf{n}}\}_{\mathbf{n}}$, in which $A_{\mathbf{n}}$ are $m$-block matrices of size $mN(\mathbf{n})\times mN(\mathbf{n})$, with $\mathbf{n}$ varying in $\mathbb{N}^k$ and $m$ \cred{being} a fixed positive integer.

\begin{definition}{\rm (Distribution notions of a matrix sequence)} 
Let $\{A_{\mathbf{n}}\}_{\mathbf{n}}$ be a matrix sequence and let $f:D\subset \R^k \rightarrow \C^{m\times m}$ be a measurable matrix-valued function defined on a set $D$ with $0<\mu_k(D)<\infty$. We write that
\begin{enumerate}
\item $\{A_{\mathbf{n}}\}_{\mathbf{n}}$ has a (asymptotic) singular value distribution described by $f$, and we write $\{A_{\mathbf{n}}\}_{\mathbf{n}}\sim_\sigma f$, if for any $F\in C_c(\mathbb{R})$ it holds
\[
\lim_{\mathbf{n}\to \infty} \frac{1}{mN(\mathbf{n})}\sum_{j=1}^{mN(\mathbf{n})}F(\sigma_j(A_{\mathbf{n}}))=\frac{1}{\mu_k(D)}\int_D \frac{1}{m} \sum_{j=1}^{m} F( \sigma_j(f(\mathbf{x}))  )\,{\rm d}\mathbf{x}.
\]
\item $\{A_{\mathbf{n}}\}_{\mathbf{n}}$ has a (asymptotic) eigenvalue (or spectral) distribution described by $f$, and we write $\{A_{\mathbf{n}}\}_{\mathbf{n}}\sim_\lambda f$, if for any $F \in C_c(\mathbb{C})$ it holds
\[
\lim_{\mathbf{n}\to \infty} \frac{1}{mN(\mathbf{n})}\sum_{j=1}^{mN(\mathbf{n})}F(\lambda_j(A_{\mathbf{n}})) = \frac{1}{\mu_k(D)}\int_D \frac{1}{m} \sum_{j=1}^{m} F( \lambda_j(f(\mathbf{x}))  )\,{\rm d}\mathbf{x}.
\]
\end{enumerate}
\end{definition}
In the case where $f$ is continuous, the spectral distribution definition expresses the idea that a suitable ordering of the eigenvalues $\{\lambda_j(A_{\mathbf{n}})\}_{j=1,\ldots,mN(\mathbf{n})}$, assigned in correspondence with an equispaced grid on $D$, approximately reconstructs the $r$
surfaces $\mathbf x$ $\mapsto\lambda_j(f(\mathbf x)),\ j=1,\ldots,m$. When $m=1$, the eigenvalues are simply approximated by a uniform sampling of $f$. A similar concept is hidden behind the singular value distribution definition.

A noteworthy example, useful for the development of our preconditioners, is that of zero-distributed matrix sequences, in which the singular value distribution is described by the constant function $f=0$.

Next, we introduce the notion of clustering, which can be seen as a special case of distribution. For $z\in\C$ and $\varepsilon>0$, denote with $B(z,\varepsilon) = \{w\in\C:\,|w-z|<\varepsilon\}$ the open disk with center $z$ and radius $\varepsilon$, while $B(S,\varepsilon)= \bigcup_{z\in S}B(z,\varepsilon)$ is the $\varepsilon$-expansion of a set $S\subseteq\C$.

\begin{definition}
Let $S\subseteq\mathbb C$ be a nonempty and closed set. A matrix sequence $\{A_{\mathbf{n}}\}_{\mathbf{n}}$ is
\begin{itemize}
    \item {\em strongly clustered} at $S$ in the sense of eigenvalues if, for any $\varepsilon>0$ and as $n\rightarrow\infty$,
        \[ \#\big\{j\in\{1,\ldots,mN(\mathbf{n})\}: \lambda_j(A_{\mathbf{n}})\notin B(S,\varepsilon)\big\} = \mathcal{O}(1). \]
    In other words, the number of eigenvalues of $A_{\mathbf{n}}$ outside of $B(S,\varepsilon)$ is bounded by a constant independent of $\mathbf{n}$;
    \item {\em weakly clustered} at $S$ in the sense of eigenvalues if, for any $\varepsilon>0$ and as $n\rightarrow\infty$,
        \[ \#\big\{j\in\{1,\ldots,mN(\mathbf{n})\}: \lambda_j(A_{\mathbf{n}})\notin B(S,\varepsilon)\big\} = o(N(\mathbf{n})). \]
    In other words, the number of eigenvalues of $A_{\mathbf{n}}$ outside of $B(S,\varepsilon)$ is negligible with respect to the size of $A_{\mathbf{n}}$.
\end{itemize}
A similar definition is given for the singular values, with $S\subseteq [0,\infty)$.
\end{definition}

In the setting of preconditioning, it is especially important the case of spectral single point clustering, where $S$ is made up by a unique element $s$. In this respect, we notice that zero-distributed matrix sequences are precisely the sequences clustered at $s=0$ in the singular value sense, as a consequence of the following result \cred{\cite[Theorem 2.16]{GaroniCapizzano_four}}.

\begin{proposition} \label{prop:distrib-cluster-equiv}
    If $\{A_n\}_n\sim_\lambda\psi$, then $\{A_n\}_n$ is weakly clustered at the essential range of $\psi$ in the sense of the eigenvalues. Furthermore, if the essential range of $\psi$ is the singleton $\{s\}$ with $s\in\C$ fixed, 
then $\{A_n\}_n\sim_\lambda\psi$ if and only if $\{A_n\}_n$ is weakly clustered at $s$ in the sense of the eigenvalues. 
The same holds in the setting of singular values, with obvious minimal changes.
\end{proposition}

We conclude this subsection with the notion of approximating class of sequences and a key related result.

\begin{definition}{\rm (Approximating class of sequences)}
	Let $\{A_{\mathbf{n}}\}_{\mathbf{n}}$ be a matrix sequence and let $\{\{B_{{\mathbf{n}},j}\}_{\mathbf{n}}\}_j$ be a class of matrix sequences. We say that $\{\{B_{{\mathbf{n}},j}\}_{\mathbf{n}}\}_j$ is an \emph{approximating class of sequences (a.c.s.)} for $\{A_{\mathbf{n}}\}_{\mathbf{n}}$ if the following condition is met: for every $j$ there exist $\mathbf{n}_j, c(j), \omega(j)$ such that, for $  \mathbf{n}\ge\mathbf{n}_j$,
	\[ A_{\mathbf{n}}=B_{{\mathbf{n}},j}+R_{{\mathbf{n}},j}+N_{{\mathbf{n}},j} \]
    with
	\[ \rank R_{{\mathbf{n}},j}\leq c(j)N(\mathbf{n}), \qquad \|N_{\mathbf{n},j}\|\leq\omega(j), \]
	where $\mathbf{n}_j$, $c(j)$ and $\omega(j)$ depend only on $j$ and \[\lim_{j\to\infty}c(j)=\lim_{j\to\infty}\omega(j)=0.\]
\end{definition}
The definition above is enclosed in the notation
$\{\{B_{{\mathbf{n}},j}\}_{\mathbf{n}}\}_j \xrightarrow{\text{a.c.s.\ wrt\ $j$}}\{A_{\mathbf{n}}\}_{\mathbf{n}}$. The a.c.s. idea gives rise to a notion of convergence expressed by the following lemma (see Proposition 2.3 in \cite{algebra_Serra-Capizzano_2001} and also \cite{Tilli_Loc,GaroniCapizzano_one,GaroniCapizzano_two} for this and more general tools).
	
\begin{lemma}\label{lem:Corollary5.1}
Let $\{A_{\mathbf{n}}\}_{\mathbf{n}}$ be a matrix sequence and let $\{\{B_{{\mathbf{n}},j}\}_{\mathbf{n}}\}_j$ be a class of matrix sequences. Let $f,f_j:D \subset \mathbb{R}^k \to \mathbb{C}$ be measurable functions defined on a set $D$ with $0<\mu_k(D)<\infty$. Suppose that
\begin{enumerate}
    \item $\{\{B_{{\mathbf{n}},j}\}_{\mathbf{n}}\}_j \sim_{\sigma}  f_j$ for every $j$,
	\item $\{\{B_{{\mathbf{n}},j}\}_{\mathbf{n}}\}_j \xrightarrow{\text{a.c.s.\ wrt\ $j$}} \{A_{\mathbf{n}}\}_{\mathbf{n}}$,
	\item $f_j \to f$ in measure.
\end{enumerate}
Then,
\[ \{A_{\mathbf{n}}\}_{\mathbf{n}} \sim_{\sigma}  f. \]
Moreover, in the case where all the involved matrices are Hermitian, if the first assumption is replaced by $\{\{B_{{\mathbf{n}},j}\}_{\mathbf{n}}\}_j \sim_{\lambda} f_j$ for every $j$ and the other two are left unchanged, then  $\{A_{\mathbf{n}}\}_{\mathbf{n}} \sim_{\lambda}  f$.
\end{lemma}

With reference to Lemma \ref{lem:Corollary5.1}, a weakening of the assumption on the Hermitian matrices for the eigenvalue distribution can be considered, as shown in \cite{barb non-h}. This is important e.g. when a \cred{first-order} differential operator is added to a \cred{second-order} operator.

\subsection{Multilevel block Toeplitz matrices}

We summarize the definition of multilevel block Toeplitz matrices and their most relevant properties. Consider the Banach space $L^1([-\pi,\pi]^k,\mathbb{C}^{m \times m})$ of all $m$-by-$m$ matrix-valued Lebesgue integrable functions over $[-\pi,\pi]^k$, equipped with the norm
\[
\|f\|_{L^1} = \frac{1}{(2\pi)^k}\int_{[-\pi,\pi]^k} \|f(\boldsymbol{\theta})\|_{\text{tr}}\,{\rm d} \boldsymbol{\theta} < \infty.
\]
Given $f:$~$[-\pi,\pi]^k\to \mathbb{C}^{m \times m}$ belonging to $L^1([-\pi,\pi]^k,\mathbb{C}^{m \times m})$, we denote with
\begin{equation}\label{multilevel_fenefun} \widehat{\mathbf{f}}_{\mathbf{j}}:=\frac{1}{(2\pi)^k}\int_{[-\pi,\pi]^k}f(\boldsymbol{\theta}){\rm e}^{\iota \left\langle { \bf j},\boldsymbol{\theta}\right\rangle}\, {\rm d}\boldsymbol{\theta},
\qquad \left\langle { \bf j},\boldsymbol{\theta}\right\rangle=\sum_{t=1}^kj_t\theta_t, \quad \iota^2=-1,
\end{equation}
its Fourier coefficients. They are square matrices of size $m$, where the integrals in (\ref{multilevel_fenefun}) are computed componentwise. The $\mathbf{n}$-th multilevel block Toeplitz matrix associated to $f$ is defined as
\begin{equation*}
    T_{\mathbf{n}}(f) := \begin{bmatrix}
        \widehat{\mathbf{f}}_{\mathbf{i-j}}
    \end{bmatrix}_{\mathbf{i,j=1}}^{\mathbf{n}}
\end{equation*}
and has size $mN(\mathbf{n})\times mN(\mathbf{n})$ . An equivalent expression is
\begin{equation*}
T_{\mathbf{n}}(f) =\sum_{|j_1|<n_1}\ldots \sum_{|j_k|<n_k} J_{n_1}^{j_1} \otimes \cdots\otimes J_{n_k}^{j_k} \otimes \widehat{\mathbf{f}}_{\mathbf{j}}, \qquad \mathbf{j}=(j_1,j_2,\dots,j_k)\in \mathbb{Z}^k,
\end{equation*}
where $J^{j}_{n}$ is the $n \times n$ matrix whose $(l,h)$-th entry equals 1 if $(l-h)=j$ and $0$ otherwise. This defines a linear operator $T_{\mathbf{n}}: L^1([-\pi,\pi]^k,\mathbb{C}^{m \times m}) \longrightarrow \mathbb{C}^{mN(\mathbf{n})\times mN(\mathbf{n})}$. It is easy to prove (see e.g. \cite{MR2108963,MR2376196,Chan:1996:CGM:240441.240445}) that if $f$ is Hermitian, then $T_{\mathbf{n}}(f)$ is Hermitian; if $f$ is Hermitian and nonnegative, but not identically zero almost everywhere, then $T_{\mathbf{n}}(f)$ is Hermitian positive definite; if $f$ is Hermitian and even, $T_{\mathbf{n}}(f)$ is real symmetric. Moreover, it holds that $T_{\mathbf{n}}\big(f(\theta)\big)^{\top}=T_{\mathbf{n}}\big(f(-\theta)\big)$. More spectral and computational properties of these block structures are discussed in \cite{serra_block1,serra_block2}.

Furthermore $\{T_{\mathbf{n}}(f)\}_{\mathbf{n}}$ is the family, or sequence, of Toeplitz matrices associated with $f$, which in turn is called the \emph{generating function} of $\{T_{\mathbf{n}}(f)\}_{\mathbf{n}}$. Throughout this work, we always assume that $f\in L^1([-\pi,\pi]^k,\mathbb{C}^{m \times m})$ is periodically extended to $\mathbb{R}^k$.

The following useful result can be proven for separable generating functions.
\begin{proposition} \label{prop:separable-toeplitz}
    Let $f\in L^1\big([-\pi,\pi]^k\big)$. For any $\mi{n}\in\mathbb{N}^k$, if $f$ is separable, i.e.
    \begin{equation*}
        f = f_1 \otimes \cdots \otimes f_k,
    \end{equation*}
    then it holds
    \begin{equation*}
        T_{\mi{n}}(f) = T_{n_1}(f_1) \otimes \cdots \otimes T_{n_k}(f_k).
    \end{equation*}
\end{proposition}

We conclude this section by proving a fundamental proposition.
\begin{proposition} \label{prop:toeplitz-f-frac-g}
    Suppose that $f,g\in L^1\big([-\pi,\pi]^k\big)$ are real valued and not identically zero, with $g\geq 0$. Then, setting
    \begin{equation*}
        r = \mathrm{ess\,inf} \frac{f}{g}, \qquad R = \mathrm{ess\,sup} \frac{f}{g},
    \end{equation*}
    and assuming that $r<R$, for any $\mi{n}\in\mathbb{N}^k$ and for all $j=1,\ldots,N(\mi{n})$ it holds
    \begin{equation*}
        r < \lambda_j \left( T_\mi{n}^{-1}(g)T_\mi{n}(f) \right) < R.
    \end{equation*}
\end{proposition}
\begin{proof}
    Let $\lambda$ be an eigenvalue of the matrix $T_\mi{n}^{-1}(g)T_\mi{n}(f)$, which is well defined because $T_\mi{n}(g)$ is positive definite due to the assumption that $g\geq 0$ not identically zero. Consider
    \begin{align*}
        T_\mi{n}^{-1}(g)T_\mi{n}(f) - \lambda I_\mi{n}
        &= T_\mi{n}^{-1}(g) \big( T_\mi{n}(f) - \lambda T_\mi{n}(g) \big) \\
        &= T_\mi{n}^{-1}(g) T_\mi{n}(f-\lambda g),
    \end{align*}
    which is similar to
    \begin{equation*}
        T_\mi{n}^{-\frac{1}{2}}(g) T_\mi{n}(f-\lambda g) T_\mi{n}^{-\frac{1}{2}}(g)
    \end{equation*}
    and hence shares the same eigenvalues. Supposing that $\lambda \leq r$, since $r<R$ we have $f-\lambda g\geq 0$ not identically zero. It follows that $T_\mi{n}(f-\lambda g)$ is positive definite, and so is $T_\mi{n}^{-\frac{1}{2}}(g) T_\mi{n}(f-\lambda g) T_\mi{n}^{-\frac{1}{2}}(g)$ by Sylvester's law of inertia. Hence, $T_\mi{n}^{-1}(g)T_\mi{n}(f) - \lambda I_\mi{n}$ is also positive definite, but we have reached a contradiction because we assumed that $\lambda$ is an eigenvalue of $T_\mi{n}^{-1}(g)T_\mi{n}(f)$. The same happens if we suppose $\lambda\geq R$, so we conclude that $\lambda\in (r,R)$.
\end{proof}

Proposition \ref{prop:toeplitz-f-frac-g} is somehow known. In the case where $k=1,2$ its proof is given in \cred{\cite[Theorem 2.4]{S-BIT94}} and \cred{\cite[Theorem 2.1]{extr2}}. For a general $k\ge 1$ but with the limitation that $g\equiv 1$ the result can be found in \cred{\cite[Theorem 3.1]{GaroniCapizzano_two}}. Here we \cred{have} given the general setting for the sake of completeness and the proof is in fact not depending on the dimensionality parameter $k$, as it can be already noticed in  \cred{\cite[Theorem 2.4]{S-BIT94}}. The block counterpart is also known. The reader is referred to \cred{\cite[Theorem 3.1]{serra_block1}} where $r$ is the essential infimum of the minimal eigenvalue function of $g^{-1}f$  and $R$ is the essential infimum of the minimal eigenvalue function of $g^{-1}f$, \cred{ under the assumptions that the minimal eigenvalue of $g^{-1}f$ is not essentially constant and the maximal eigenvalue of $g^{-1}f$ is also not essentially constant. In this case, in analogy with the scalar setting, the proof is given for $k=2$ but it perfectly holds for any $k\ge 1$.}

\subsection{Multilevel block GLT sequences}
In this subsection, we introduce multilevel block GLT sequences, a class of special matrix sequences. The formal constructive definition can be found in 
\cite{GaroniCapizzano_one,GaroniCapizzano_two,GaroniCapizzano_three,GaroniCapizzano_four}; for the purpose of this paper it is more convenient to present them through their main properties, listed below. They are equivalent to the full definition and constitute a complete characterisation of the GLT class. Each GLT sequence is equipped with a measurable function $\kappa:[0,1]^k\times[-\pi,\pi]^k\to\mathbb C^{m\times m}$, called \emph{symbol}, which is essentially unique, in the sense that if $\kappa$ and $\varsigma$ are symbols for the same sequence, then $\kappa=\varsigma$ a.e.. The notation $\{A_{\mathbf{n}}\}_{\mathbf{n}}\sim_{\rm
GLT}\kappa$ indicates that $\{A_{\mathbf{n}}\}_{\mathbf{n}}$ is a GLT sequence with symbol $\kappa$.

\begin{enumerate}
    \item[\textbf{GLT\,1.}] If $\{A_{\mathbf{n}}\}_{\mathbf{n}}\sim_{\rm GLT}\kappa$ then $\{A_{\mathbf{n}}\}_{\mathbf{n}}\sim_\sigma\kappa$. Moreover, if each $A_{\mathbf{n}}$ is Hermitian then \cred{$\{A_{\mathbf{n}}\}_{\mathbf{n}}\sim_\lambda\kappa$ ;}
    \item[\textbf{GLT\,2.}]
    \begin{itemize}
        \item If $f\in L^1([-\pi,\pi]^k,\mathbb{C}^{m \times m})$, then $\{T_{\mathbf{n}}(f)\}_\mathbf{n}\sim_{\rm GLT} f(\mathbf{\theta})$;
        \item Given $a:[0,1]^k\to\mathbb C^{m\times m}$, we define the multilevel block diagonal sampling matrix $D_{(\mathbf{n},m)}(a)\in\mathbb C^{m N(\mathbf{n})\times m N(\mathbf{n})}$ as
        \[ D_{\mathbf{n}}(a):=\mathop{\rm diag}_{\mathbf{i}=\mathbf{1},\ldots,\mathbf{n}}a\Bigl(\frac{\mathbf{i}}{\mathbf{n}}\Bigr). \]
        If $a$ is Riemann-integrable, $\{D_{\mathbf{n}}(a)\}_\mathbf{n}\sim_{\rm GLT}a(\mathbf{x})$;
        \item $\{Z_{\mathbf{n}}\}_\mathbf{n}\sim_{\rm GLT} O_m$ if and only if $\{Z_{\mathbf{n}}\}_\mathbf{n}\sim_\sigma 0$.
    \end{itemize}
    \item[\textbf{GLT\,3.}] If $\{A_{\mathbf{n}}\}_{\mathbf{n}}\sim_{\rm GLT}\kappa$ and $\{B_{\mathbf{n}}\}_{\mathbf{n}}\sim_{\rm GLT}\varsigma$, then
    \begin{itemize}
        \item $\{A_{\mathbf{n}}^*\}_{\mathbf{n}}\sim_{\rm GLT}\kappa^*$;
        \item $\{\alpha A_{\mathbf{n}}+\beta B_{\mathbf{n}}\}_\mathbf{n}\sim_{\rm GLT}\alpha\kappa+\beta\varsigma$ for any $\alpha,\beta\in\mathbb C$;
        \item $\{A_{\mathbf{n}}B_{\mathbf{n}}\}_\mathbf{n}\sim_{\rm GLT}\kappa\varsigma$;
        \item Denoting with $A^\dag$ the Moore--Penrose pseudoinverse of a matrix $A$ (recall that $A^\dag=A^{-1}$ whenever $A$ is invertible), $\{A_{\mathbf{n}}^\dag\}_\mathbf{n}\sim_{\rm GLT}\kappa^{-1}$ provided that $\kappa$ is invertible a.e..
    \end{itemize}
    \item[\textbf{GLT\,4.}] If $\{A_{\mathbf{n}}\}_{\mathbf{n}}\sim_{\rm GLT}\kappa$ and each $A_{\mathbf{n}}$ is Hermitian, then $\{f(A_{\mathbf{n}})\}_{\mathbf{n}}\sim_{\rm GLT}f(\kappa)$ for every continuous function $f\colon\C\rightarrow\C$.
    \end{enumerate}

We emphasize that the GLT construction is effective for the analysis of different types of problems, ranging from the distribution of orthogonal polynomial zeros \cite{kuij}, to the spectral analysis of non-normal structures arising in imaging \cite{ngondiep}, to the spectral/singular value analysis of discretized PDEs and systems of PDEs (see \cite{GaroniCapizzano_four,systemPDE} and references therein).

\subsection{Multilevel \texorpdfstring{$\tau$}m
matrices} \label{ssec:tau-matrices}

Before introducing our preconditioner, we recall a few basic definitions concerning $\tau$ matrices (see \cite{Bini1990}). The algebra of unilevel real $\tau$ matrices of size $n$ is the set of real matrices simultaneously diagonalized by the sine transform orthogonal matrix $Q_n := \left[ \sqrt{\frac{2}{n+1}} \sin \left( \frac{\pi i j}{n+1} \right) \right]_{i,j=1}^n$, i.e., the set $\{Q_nD_nQ_n : D_n\in\R^{n\times n} \text{ diagonal}\}$. The algebra of multilevel $\tau$ matrices is defined in the same way from the multilevel sine transform matrix $Q_\mathbf{n} := Q_{n_1}\otimes\ldots\otimes Q_{n_k}$.
The absolute value of a $\tau$ matrix $S_\mathbf{n}$ is given by $\abs{S_\mathbf{n}} = \left(S_\mathbf{n}^{\top} S_\mathbf{n}\right)^\frac{1}{2} = Q_\mathbf{n}^{\top} \abs{D_\mathbf{n}} Q_\mathbf{n}$, where $\abs{D_\mathbf{n}}$ is the diagonal matrix in which all entries of $D_\mathbf{n}$ are replaced with their absolute value. The \cred{unilevel} $\tau$ preconditioner of a real symmetric Toeplitz matrix $T_n = \left[ t_{\abs{i-j}}\right]_{i,j=1}^n$ is defined as $\tau(T_n):= T_n - H_n(T_n)$, where $H_n(T_n)$ is the Hankel matrix whose antidiagonals are constant and equal to $t_2,\ldots,t_{n-1},0,0,0,t_{n-1},\ldots,t_2$. Note that, if $T_n=T_n(f)$ with $f$ any Lebesgue integrable function, it holds $\left\{T_n\right\}_n\sim_{\rm GLT} f$ by GLT 2. Then, the related $\tau$ preconditioners sequence $\left\{\tau(T_n)\right\}_n$ has the same symbol, due to GLT 3 and \cred{the fact} that $\left\{H_n(T_n)\right\}_n$ has been proven to have singular value symbol zero \cite{FaTi} and hence GLT symbol zero by GLT 2.

The computational power of $\tau$ preconditioners relies on the diagonalization property: the inverse of any $\tau$
matrix is immediately obtained by inverting the diagonal part, while the matrix-vector product can be performed via the discrete sine transforms (DSTs) with the cost of $\mathcal{O}(n \log n)$ operations \cite{BC83}.

\section{Spectral properties of symmetrized multilevel Toeplitz matrices} \label{sec:properties-symmetrized}
In this section we outline the available spectral results on symmetrized Toeplitz matrices, obtained by permuting the rows of a standard Toeplitz matrix.

More precisely, the anti-identity permutation matrix $Y_n\in\R^{n\times n}$ is defined as
\begin{equation*}
    Y_n := \begin{bmatrix}
        & & 1 \\ & \iddots \\ 1
    \end{bmatrix}.
\end{equation*}
Then, left multiplying a unilevel \cred{nonsymmetric} Toeplitz matrix $T_n\in\C^{n\times n}$ by $Y_n$ results in a symmetric matrix, or more precisely a Hankel matrix. The multilevel counterpart of the anti-identity matrix is given by
\begin{equation*}
    Y_\mathbf{n} := Y_{n_1}\otimes\cdots\otimes Y_{n_k} = Y_{n_1 \cdots n_k}.
\end{equation*}
It is easy to check that the product $Y_\mathbf{n}T_\mathbf{n}$, where $T_\mathbf{n}\in\C^{N(\mathbf{n})\times N(\mathbf{n})}$ is a multilevel \cred{nonsymmetric} Toeplitz matrix $T_\mathbf{n}\in\C^{N(\mathbf{n})\times N(\mathbf{n})}$, is in fact symmetric.

Precise distributional results for this type of symmetrized Toeplitz structures are given in \cite{Ferrari2019}, in the unilevel case, and in \cite{Ferrari2021}, in the multilevel one. We report here the main theorem on which we base the development of further theory.

\begin{theorem}\cite[Corollary 3.5]{Ferrari2021} \label{thm:symm-mat-distr-multilev}
Suppose $f\in L^1([-\pi,\pi]^k)$ is a $k$-variate function with real Fourier coefficients, periodically extended to $\R^k$. Let $Y_{\mathbf{n}} = Y_{n_1}\otimes \ldots \otimes Y_{n_k} = Y_{N(\mathbf{n})}\in\R^{N(\mathbf{n})\times N(\mathbf{n})}$ be the multilevel anti-identity matrix and let $T_{\mathbf{n}}(f)\in\R^{N(\mathbf{n})\times N(\mathbf{n})}$ be the multilevel Toeplitz matrix generated by $f$. Then
\begin{equation*}
    \left\{ Y_{\mathbf{n}}T_{\mathbf{n}}(f) \right\}_{\mathbf{n}} \sim_\lambda \phi_{\abs{f}},
\end{equation*}
over the domain \cred{$[-2\pi,0 )^k \cup [0,2\pi]^k$} where $\phi_g$ is defined in the following way
\begin{equation*}
    \phi_g := \left\{
    \begin{array}{cl}
        g(\boldsymbol{\theta}), & \boldsymbol{\theta}\in [0,2\pi]^k, \\
        -g(-\boldsymbol{\theta}), & \boldsymbol{\theta}\in  [-2\pi,0)^k.
    \end{array} \right.\,
\end{equation*}
\end{theorem}

The spectral information is used in \cite{Ferrari2019} to construct efficient circulant preconditioners for the related symmetrized matrix sequences. These results can be tailored with minor adjustments to the setting of $\tau$ preconditioners. However, in what follows, we give general results for a GLT preconditioning which contain the corresponding adaptation when the preconditioner is $\tau$.

\begin{theorem} \label{thm:precond-toeplitz-seq-distrib}
    Let $f\in L^1([-\pi,\pi]^k)$ with real Fourier coefficients and sparsely vanishing i.e. with a set of zeros of zero Lebesgue measure. Suppose that $P_\mi{n}$ are real \cred{centrosymmetric} positive definite preconditioners such that
 \[    
 \{ P_\mi{n}^{-1} T_\mi{n}(f) \}_\mi{n} \sim_{\rm GLT} u
\]
with $u\in L^1([-\pi,\pi]^k)$ unimodular function i.e. $|u(\boldsymbol{\theta})|=1$ for every $\boldsymbol{\theta}$ and with $u$ having real Fourier coefficients. 
Then  $\{ P_\mi{n}^{-1} Y_\mi{n} T_\mi{n}(f) \}_\mi{n}$ is clustered in the eigenvalue sense at $\pm 1$ and     
        \begin{equation*}
        \{ \cred{P_\mi{n}^{-1} Y_\mi{n} T_\mi{n}(f)} \}_\mi{n} \sim_\lambda \phi_1(\boldsymbol{\theta}) := \left\{
        \begin{array}{cl}
            1, & \boldsymbol{\theta}\in [0,2\pi]^k, \\
        -1, & \boldsymbol{\theta}\in  [-2\pi,0)^k.
        \end{array} \right.
    \end{equation*}
\end{theorem}
\begin{proof}
We first recall that any \cred{centrosymmetric} matrix commutes with $Y_\mi{n}$ and that the inverse and the square root (the principal one) of a \cred{centrosymmetric} matrix is still \cred{centrosymmetric} 
(see \cite{centrosy3,centrosy2,centrosy1} and references therein; in particular \cred{\cite[Corollary 1]{centrosy3}}). In our setting $P_\mi{n}$ is real \cred{centrosymmetric} and positive definite so that 
its principal square root is real centrosymmetric and it is the positive definite square root obtained via the Schur decomposition.

Now we consider the Hermitian matrix $P_\mi{n}^{-1/2} Y_\mi{n} T_\mi{n}(f)P_\mi{n}^{-1/2}$ which is similar to our preconditioned matrix $P_\mi{n}^{-1} Y_\mi{n} T_\mi{n}(f)$ so that they have the same eigenvalues. With these premises we have
\[
P_\mi{n}^{-1/2} Y_\mi{n} T_\mi{n}(f)P_\mi{n}^{-1/2}=Y_\mi{n} P_\mi{n}^{-1/2} T_\mi{n}(f)P_\mi{n}^{-1/2}.
\]
By Theorem \ref{thm:symm-mat-distr-multilev} we now that $\{ Y_\mi{n} T_\mi{n}(f) \}_\mi{n}$ is real symmetric and distributed as $\phi_{\abs{f}}$ and hence $N(\mi{n})+o(N(\mi{n}))$ eigenvalues of $P_\mi{n}^{-1/2} Y_\mi{n} T_\mi{n}(f)P_\mi{n}^{-1/2}$ are negative and $N(\mi{n})+o(N(\mi{n}))$ eigenvalues of $P_\mi{n}^{-1/2} Y_\mi{n} T_\mi{n}(f)P_\mi{n}^{-1/2}$ are positive due to Sylvester inertia law.

Now we use the powerful GLT tools. By the third part of item GLT 3., the assumption $\{ P_\mi{n}^{-1} T_\mi{n}(f) \}_\mi{n} \sim_{\rm GLT} u$ and the fact that $\{ T_\mi{n}(f) \}_\mi{n} \sim_{\rm GLT} f$ by the first part of item GLT 2., imply that $\{ P_\mi{n}^{-1} \}_\mi{n} \sim_{\rm GLT} u/f$ with $u/f$ nonnegative because $P_\mi{n}$ is positive definite and with $u/f$ sparsely vanishing because of the assumptions on $f$ and $u$. Hence by item  GLT 4. 
$\{ P_\mi{n}^{-1/2} \}_\mi{n} \sim_{\rm GLT} \sqrt{u/f}$ and therefore, again by the third part of item GLT 3. we deduce
\[
\{ P_\mi{n}^{-1/2} T_\mi{n}(f)P_\mi{n}^{-1/2} \}_\mi{n} \sim_{\rm GLT} u.
\]
Since $u$ is unimodular, item GLT 1. leads to the conclusion that $\{ P_\mi{n}^{-1/2} T_\mi{n}(f)P_\mi{n}^{-1/2} \}_\mi{n}$ is clustered at $1$ in the singular value sense and the same is true for the whole symmetrized preconditioned matrix sequence since $Y_\mi{n}$ is unitary.
However, the symmetrized preconditioned matrix sequence is Hermitian and hence the moduli of the eigenvalues are equal to the singular values and the eigenvalues are real.

From the latter we directly deduce that 
\[
  \{P_\mi{n}^{-1/2} Y_\mi{n} T_\mi{n}(f)P_\mi{n}^{-1/2}\}_\mi{n} = \{ Y_\mi{n} P_\mi{n}^{-1/2} T_\mi{n}(f)P_\mi{n}^{-1/2} \}_\mi{n} 
  \] 
is clustered in the eigenvalue sense at $\pm 1$. However, as already observed $N(\mi{n})+o(N(\mi{n}))$ eigenvalues of $P_\mi{n}^{-1/2} Y_\mi{n} T_\mi{n}(f)P_\mi{n}^{-1/2}$ are negative and $N(\mi{n})+o(N(\mi{n}))$ eigenvalues of $P_\mi{n}^{-1/2} Y_\mi{n} T_\mi{n}(f)P_\mi{n}^{-1/2}$ are positive and therefore we finally obtain 
\[
\{ \cred{P_\mi{n}^{-1} Y_\mi{n} T_\mi{n}(f)} \}_\mi{n} \sim_\lambda \phi_1
\]
and the claimed thesis is proven.
\end{proof}

In the following remark we discuss the assumptions in more detail.

\begin{remark} \label{remark:acs-cluster-equiv}
The assumption on the GLT character of the preconditioned matrix sequence can be stated in terms of the a.c.s notion.

For instance $\{ P_\mi{n}^{-1} T_\mi{n}(f) \}_\mi{n} \sim_{\rm GLT} u$ in terms of a.c.s approximation can be written as
$\{ \{ T_\mi{n}(u) \}_\mi{n} \}_m$ is an a.c.s.\! for $\{ P_\mi{n}^{-1} T_\mi{n}(f) \}_\mi{n}$ or $\{ \{ T_\mi{n}(u_m) \}_\mi{n} \}_m$ is an a.c.s.\! for $\{ P_\mi{n}^{-1} T_\mi{n}(f) \}_\mi{n}$
and $u_m$ converges to $u$ as in hypothesis 3. of Lemma \ref{lem:Corollary5.1}. There are practical situations in which proving an a.c.s relation is easier.

\end{remark}

\begin{remark} \label{remark:circ-tau-matrix-alg-prec}
  When considering circulant and $\tau$ preconditioners there are slight variations.
  For instance, given $f$ as in Theorem \ref{thm:precond-toeplitz-seq-distrib} we can construct circulant preconditioners $S_\mi{n}$ such $\{ S_\mi{n}^{-1} T_\mi{n}(f) \}_\mi{n}\sim_{\rm GLT} 1$, while it is impossible to construct $\tau$ preconditioners $S_\mi{n}$ such $\{ S_\mi{n}^{-1} T_\mi{n}(f) \}_\mi{n}\sim_{\rm GLT} 1$, unless $f$ is real-valued almost everywhere. However, both in the case of circulant and $\tau$ preconditioners there are canonical ways to construct $P_\mi{n}$ positive definite and \cred{centrosymmetric} with 
  $\{ P_\mi{n}^{-1} T_\mi{n}(f) \}_\mi{n}\sim_{\rm GLT} f/|f|$. We observe that $f$ with real Fourier coefficient implies that $f/|f|$ has real Fourier coefficients and hence all the assumptions of Theorem \ref{thm:precond-toeplitz-seq-distrib} are fulfilled. 
   
  
\end{remark}

\begin{remark} \label{remark:role of centro-symmetry}
  The role of \cred{centrosymmetry} of the preconditioners is important. Indeed if we drop such assumption and we maintain the rest of the hypotheses in Theorem \ref{thm:precond-toeplitz-seq-distrib}, then by following a similar path as in its proof we arrive to write that
\[  
  \{Y_\mi{n} T_\mi{n}(f)P_\mi{n}^{-1}\}_\mi{n} = \{Y_\mi{n} T_\mi{n}(u)\}_\mi{n} + \{Y_\mi{n}\left[ T_\mi{n}(f)P_\mi{n}^{-1}-T_\mi{n}(u)\right]\}_\mi{n}
  \]
  with $Y_\mi{n} T_\mi{n}(f)P_\mi{n}^{-1}$ similar to the preconditioned matrix sequence $P_\mi{n}^{-1}Y_\mi{n} T_\mi{n}(f)$.
  Now by Theorem \ref{thm:symm-mat-distr-multilev} we have
\[
\{ \cred{Y_\mi{n} T_\mi{n}(u)} \}_\mi{n} \sim_\lambda \phi_1
\]  
since $|u|=1$ and  $\{Y_\mi{n} \left[T_\mi{n}(f)P_\mi{n}^{-1}-T_\mi{n}(u)\right]\}_\mi{n}\sim_{\rm GLT} 0$ because of the GLT axioms.
However, we cannot conclude $\{ P_\mi{n}^{-1}Y_\mi{n} T_\mi{n}(f) \}_\mi{n} \sim_\lambda \phi_1$ because the zero-distributed perturbation $\{Y_\mi{n} \left[ T_\mi{n}(f)P_\mi{n}^{-1}-T_\mi{n}(u)\right]\}_\mi{n}\sim_{\rm GLT}$ is not Hermitian.
However the matrix sequence $\{Y_\mi{n} T_\mi{n}(u)\}_\mi{n}$ is real symmetric and hence a weaker result than in Theorem \ref{thm:precond-toeplitz-seq-distrib} can be obtained under the assumptions of the main result in \cite{barb non-h}.
\end{remark}

\subsection{The block setting}

By making a minor variation in the proof of Theorem \ref{thm:symm-mat-distr-multilev}, the block setting in which $f\in L^1([-\pi,\pi]^k,\mathbb{C}^{m \times m})$ can easily be included (see \cite[Theorem 3.4]{Ferrari2019}), under the natural assumption that the Fourier coefficients of $f$ are Hermitian.

\begin{theorem} \label{thm:symm-mat-distr-multilev-block} 	
Suppose $f\in L^1([-\pi,\pi]^k,\mathbb{C}^{m \times m})$ is a $k$-variate matrix-valued function with Hermitian Fourier coefficients, periodically extended to the whole real plane. Let $Y_{\mathbf{n}}(m) := Y_{\mathbf{n}}\otimes I_m \in\R^{mN(\mathbf{n})\times mN(\mathbf{n})}$ and let $T_{\mathbf{n}}(f)\in\C^{mN(\mathbf{n})\times mN(\mathbf{n})}$ be the block multilevel Toeplitz matrix generated by $f$. Then,
\begin{equation*}
    \left\{ Y_{\mathbf{n}}T_{\mathbf{n}}(f) \right\}_{\mathbf{n}} \sim_\lambda \phi_{\abs{f}}, \qquad \abs{f}=(f^*f)^{\frac{1}{2}},
\end{equation*}
over the domain \cred{ $[-2\pi,0 )^k \cup [0,2\pi]^k$ }, where $\phi_g$ is defined in the following way
\begin{equation*}
    \phi_g := \left\{
    \begin{array}{cl}
        g(\boldsymbol{\theta}), & \boldsymbol{\theta}\in [0,2\pi]^k, \\
        -g(-\boldsymbol{\theta}), & \boldsymbol{\theta}\in  [-2\pi,0)^k.
    \end{array} \right.\,
\end{equation*}
\end{theorem}

\begin{theorem}\label{thm:precond-seq-distrib-multilev-block}
    Let $f\in L^1([-\pi,\pi]^k,\mathbb{C}^{m \times m})$ is a $k$-variate matrix-valued function with Hermitian Fourier coefficients and let $Y_{\mathbf{n}}(m)=Y_\mathbf{n}\otimes I_m \in\R^{mN(\mathbf{n})\times mN(\mathbf{n})}$. 
   Suppose that $f$ is sparsely vanishing i.e. with $\sigma_{\min}(f)$ \cred{having a set of zeros of zero Lebesgue measure.}  
    Let $T_\mathbf{n}(f)\in\R^{mN(\mathbf{n})\times mN(\mathbf{n})}$ be the block multilevel Toeplitz matrix generated by $f$. 
  Suppose that $P_\mi{n}$ are \cred{centrosymmetric} positive definite preconditioners such that
 \[    
 \{ P_\mi{n}^{-1} T_\mi{n}(f) \}_\mi{n} \sim_{\rm GLT} u
\]
with $u\in L^1([-\pi,\pi]^k),\mathbb{C}^{m \times m}$ being unitary matrix-valued i.e. $u(\boldsymbol{\theta})$ unitary matrix for almost every $\boldsymbol{\theta}$ and with $u$ having Hermitian Fourier coefficients. 
Then  $\{ P_\mi{n}^{-1} Y_\mi{n} T_\mi{n}(f) \}_\mi{n}$ is clustered in the eigenvalue sense at $\pm 1$ and 
    \[ \left\{ {P_\mathbf{n}}^{-1} Y_{\mathbf{n}}(m) T_\mathbf{n}(f) \right\}_\mathbf{n} \sim_\lambda \phi_{I_m} = \left\{
    \begin{array}{cl}
        I_m, & \boldsymbol{\theta}\in [0,2\pi]^k, \\
        -I_m, & \boldsymbol{\theta}\in  [-2\pi,0)^k.
    \end{array} \right.\, \]
\end{theorem}

Since the block $\tau$ algebra is inherently real symmetric, the generalization of Theorem \ref{thm:precond-toeplitz-seq-distrib} requires a more restrictive assumption than that in Theorem \ref{thm:symm-mat-distr-multilev-block}, namely we need that the Fourier coefficients of $f\in L^1([-\pi,\pi]^k,\mathbb{C}^{m \times m})$ are all real symmetric.

Few more remarks are in order.
\begin{remark} 
Once a distribution symbol is found, infinitely many of them exist since any rearrangement is still a symbol and those rearrangements can have a different number of variables and a different dimensionality (see e.g. \cite{non-unique-s}). For instance, with reference to Theorem \ref{thm:precond-seq-distrib-multilev-block} we also have
\[ \left\{ {P_\mathbf{n}}^{-1} Y_{\mathbf{n}}(m) T_\mathbf{n}(f) \right\}_\mathbf{n} \sim_\lambda \phi_{1} = \left\{
    \begin{array}{cl}
        1, & \boldsymbol{s}\in D_1, \\
        -1, & \boldsymbol{s}\in  D_2,
    \end{array} \right.\, \]
for any measurable sets $D_1,D_2\in\R^q$, $q\ge 1$, such that $D_1\bigcap D_2=\emptyset$ and $\mu_q(D_1)=\mu_q(D_2)\in (0,\infty)$. 
\end{remark}

\begin{remark} \label{rem:omega-circ vs tau}
A variation of Theorem \ref{thm:precond-seq-distrib-multilev-block} can be given with the weaker assumption that the Fourier coefficients are only Hermitian, as in Theorem \ref{thm:symm-mat-distr-multilev-block}, by considering a block algebra which is not inherently real symmetric. For instance, one could consider any $\omega$ block circulant algebra with $\abs{\omega}=1$ (see e.g. \cite{DS,DS-2}).
\end{remark}

\color{black}

\section{Spectral analysis and preconditioning proposal for nonsymmetric multilevel Toeplitz systems} \label{sec:spectral_nonsymmetric}

We demonstrate the effectiveness of the $\tau$ preconditioning strategy proposed in the previous section by applying it to the nonsymmetric Toeplitz linear systems that arise from the discretization of $k$-dimensional space fractional diffusion equations. Our model problem is the following \cred{R.-L.} fractional diffusion equation
\begin{equation}\label{eq:fde}
\left\{ \begin{array}{ll}
    \frac{\partial u(x,t)}{\partial t} = \sum_{i=1}^{k}\Big( d_{i,+}\frac{\partial_+^{\alpha_i}}{\partial x^{\alpha_i}} + d_{i,-}\frac{\partial_-^{\alpha_i}}{\partial x^{\alpha_i}}\Big) u(x,t) + f(x,t),
    &x\in\Omega, t\in (0,T],\\
    u(x,t) = 0, &x \in \partial\Omega, \\
    u(x,0) = u_0(x), &x\in\Omega,
    \end{array} \right.\,
\end{equation}
where $\Omega=\prod_{i=1}^{k}(a_i,b_i)$ is an open hyperrectangle in $\R^k$, $\partial\Omega$ denotes the boundary of $\Omega$, $\alpha_i \in (1,2)$ are the fractional derivative orders, $f(x, t)$ is the source term, and the diffusion coefficients $d_{i,\pm}$ are nonnegative constants. \cred{The R.-L.} fractional derivatives in \eqref{eq:fde} are defined by
\begin{equation*}
\begin{split}
\frac{\partial_+^{\alpha_i} u(x,t)}{\partial x^{\alpha_i}}&:=\frac{1}{\Gamma(2-{\alpha_i})}\frac{\partial^{2}}{\partial x_{i}^{2}}\int_{a_i}^{x_i}\frac{u(x_{1},x_{2},\dots,x_{i-1},\xi,x_{i+1},\dots,x_{m},t)}{(x_{i}-\xi)^{\alpha_{i}-1}}{\rm d}\xi, \\
\frac{\partial_-^{\alpha_i} u(x,t)}{\partial x^{\alpha_i}}&:=\frac{1}{\Gamma(2-{\alpha_i})}\frac{\partial^{2}}{\partial x_{i}^{2}}\int_{x_i}^{b_i}\frac{u(x_{1},x_{2},\dots,x_{i-1},\xi,x_{i+1},\dots,x_{m},t)}{(\xi-x_i)^{\alpha_i-1}}{\rm d}\xi,
\end{split}
\end{equation*}
respectively, where ${\Gamma(\cdot)}$ denotes the gamma function.

We follow the discretization performed in \cite{Meerschaert2004}, combining an implicit Euler method for the time derivative with the shifted Gr\"unwald scheme for the spatial fractional derivatives. The resulting scheme has \cred{first-order} accuracy in both the time and space directions. To simplify the presentation of our preconditioning idea, we begin by discussing the one-dimensional case before advancing to the multi-dimensional scenario.

The one-dimensional version of \eqref{eq:fde} corresponds to setting \cred{ $k=1$;} therefore we can momentarily neglect the index $i$ in the notation. In what follows, definitions and results from Section \ref{sec:prelim} are applied with $\mathbf{n}=n$ a standard index and $m=1$. For $n,l\in\mathbb{N}$ we define the space step size $h:=\frac{b-a}{n+1}$, the time step size $\tau:=\frac{T}{l}$ and the coefficient $\nu:=\frac{h^\alpha}{\tau}$, where $\alpha$ is the fractional derivative order. The aforementioned numerical scheme leads to the linear systems\cred{
\begin{equation*} 
    \Big[ \nu I_n + d_+ T_n(v_\alpha) + d_- T_n(v_\alpha)^\top \Big] \vet{u}^{(l)}
    = \nu \vet{u}^{(l-1)} + h^\alpha \vet{f}^{(l)},
\end{equation*}}
in which $\mathbf{f}^{(l)} \in \mathbb{R}^n$ is the constant term known from the numerical scheme used and\cred{
\begin{equation} \label{eq:v_alpha-def}
    v_\alpha(\theta) := -e^{-\bf i\theta}\left(1-e^{\bf i\theta}\right)^\alpha
\end{equation}
with ${\bf i}$ denoting the imaginary unit.}

It is known from \cite{donatelli161} that $T_n(v_\alpha)$ is the Toeplitz matrix
\begin{equation*}
    T_n(v_\alpha) := - \begin{bmatrix}
        \omega_1^{(\alpha)} & \omega_0^{(\alpha)} & \\
        \omega_2^{(\alpha)} & \omega_1^{(\alpha)} & g_0^{(\alpha)} & \\
        \omega_3^{(\alpha)} & \omega_2^{(\alpha)} & \omega_1^{(\alpha)} & \omega_0^{(\alpha)} & \\
        \vdots & \ddots & \ddots & \ddots & \ddots \\
        \vdots & \ddots & \ddots & \ddots & \ddots & \ddots \\
        \omega_{n-1}^{(\alpha)} & \omega_{n-2}^{(\alpha)} & \cdots & \ddots & \omega_2^{(\alpha)} & \omega_1^{(\alpha)} & \omega_0^{(\alpha)} \\
        \omega_n^{(\alpha)} & \omega_{n-1}^{(\alpha)} & \cdots & \cdots & \omega_3^{(\alpha)} & \omega_2^{(\alpha)} & \omega_1^{(\alpha)} \\
     \end{bmatrix},
\end{equation*}
where the entries are real and given by
\begin{align*}
    \omega_0^{(\alpha)} &:= 1; \\
    \omega_k^{(\alpha)} &:= (-1)^k \binom{\alpha}{k}
    = \frac{(-1)^k}{k!} \alpha(\alpha-1)\cdots(\alpha-k+1), \qquad k\geq 1.
\end{align*}
Being a sum of Toeplitz matrices, the coefficient matrix has Toeplitz structure and is clearly nonsymmetric whenever $d_+\neq d_-$. More precisely, by linearity of the Toeplitz operator and the properties of Toeplitz matrices it is easy to see that
\begin{equation*}
    \nu I_n + d_+ T_n(v_\alpha) + d_- T_n(v_\alpha)^\top = T_n(g),
\end{equation*}
where $g$ is the following complex-valued function with real Fourier coefficients
\begin{equation} \label{eq:g-1d-def}
    g(\theta) := \nu + d_+ v_\alpha(\theta) + d_- v_\alpha(-\theta) = \nu + d_+ v_\alpha(\theta) + d_- \overline{v_\alpha(\theta)}\,.
\end{equation}
Then the linear system of interest becomes
\begin{equation} \label{eq:linsyst-1d}
     T_n(g) \vet{u}^{(l)}
    = \nu \vet{u}^{(l-1)} + h^\alpha \vet{f}^{(l)}.
\end{equation}

Following the idea in \cite{doi:10.1137/140974213}, we symmetrize this linear system by premultiplying it with the anti-identity matrix $Y_n$. We obtain a linear system with symmetric coefficient matrix $Y_n T_n[g]$, to which the MINRES method can be applied. In order to construct efficient preconditioners, we are interested in the spectral distribution of the sequence $\{Y_nT_n(g)\}_n$. Applying Theorem \ref{thm:symm-mat-distr-multilev} we find that
\begin{equation*}
    \{Y_nT_n(g)\}_n \sim_\lambda
    \phi_{\abs{g}}(\theta)=\left\{
    \begin{array}{cl}
        \abs{g(\theta)}, & \theta\in [0,2\pi], \\
        -\abs{g(-\theta)}, & \theta\in  [-2\pi,0),
    \end{array} \right.\,
\end{equation*}
in which

\begin{equation} \label{eq:g-abs-calc}
\begin{split}
    \abs{g}^2 &= g\overline{g}
    = \Big[ \nu + d_+v_\alpha + d_-\overline{v_\alpha} \Big] \Big[ \nu + d_+ \overline{v_\alpha} + d_- v_\alpha \Big] \\
    &= \nu^2 + \nu \big(d_++d_-\big) \big(v_\alpha +
        \overline{v_\alpha} \big)  + \big(d_+^2+d_-^2\big)\abs{v_\alpha}^2 + d_+d_- \big( v_\alpha^2 + \overline{v_\alpha}^2 \big) \\
    &= \nu^2 + \nu \big(d_++d_-\big) \big(v_\alpha + \overline{v_\alpha} \big)  + \big(d_+ - d_-\big)^2 \abs{v_\alpha}^2 + d_+ d_- \big( v_\alpha + \overline{v_\alpha} \big)^2,
\end{split}
\end{equation}
where we exploited the identity $(a^2+b^2)\abs{z}^2 + ab(z^2+\overline{z}^2) = (a-b)^2\abs{z}^2 + ab(z+\overline{z})^2$, which is verified for any $z\in\C$ and $a,b\in\R$. Observing that $\abs{v_\alpha}^2 = (2-2\cos\theta)^\alpha$, we conclude
\begin{align*}
    \abs{g} = \Big[ \nu^2 + \nu \big(d_++d_-\big) \big(v_\alpha + \overline{v_\alpha} \big) +  (d_+ - d_-)^2 (2-2\cos{\theta})^\alpha + d_+ d_- ( v_\alpha + \overline{v_\alpha} )^2 \Big]^\frac{1}{2}.
\end{align*}

We build our preconditioner on the spectral information retrieved above. The idea is that the sequence of the preconditioners is designed to match the symbol $\abs{g}$, in order to balance the eigenvalues of the coefficient matrix as the size $n$ grows and tends to infinity. As a matter of fact, we will prove that this strategy leads to a preconditioned matrix sequence clustered at $\pm 1$. The novel preconditioner that we propose is

\begin{equation} \label{eq:precond-def-1d}
\begin{split}
    P_n: = \Big[ &\nu^2 I_n + \nu\big(d_+ + d_-\big) \tau \big( T_n(v_\alpha) + T_n(v_\alpha)^\top \big) \\
    &+ \big(d_+ - d_-\big)^2 \big(T_n(2-2\cos{\theta})\big)^\alpha
    + d_+ d_- \tau \big( T_n(v_\alpha) + T_n(v_\alpha)^\top \big)^2 \Big]^{\frac{1}{2}},
\end{split}
\end{equation}

in which
\begin{eqnarray*}
T_n(2-2\cos{\theta})=\begin{bmatrix}{}
2 & -1 & & &  \\
-1 & 2 & -1 & &   \\
& \ddots& \ddots &  \ddots &\\
& &\ddots& \ddots &  -1 \\
& &  & -1 & 2
\end{bmatrix}\in \mathbb{R}^{n \times n}.
\end{eqnarray*}
It is well known that $T_n(2-2\cos{\theta})$ is a $\tau$ matrix\cred{; therefore,} it can be diagonalized by $Q_n$. The same holds for $\tau \big( T_n(v_\alpha) + T_n(v_\alpha)^\top \big)$. Using the algebra properties of the set of $\tau$ matrices, it is immediate to see that $P_n$ is a $\tau$ matrix as well, as combination of $\tau$ matrices. Then the preconditioning step $P_n^{-1}\mathbf{d}$, with $\mathbf{d}$ any vector, can be efficiently implemented via the DST in $\mathcal{O}(n \log n)$ operations \cite{BC83}.

In order to reveal the effectiveness of $P_n$ as preconditioner for \eqref{eq:linsyst-1d}, we prove that the condition in Remark \ref{remark:acs-cluster-equiv} holds for $\left\{ P_n^{-1} T_n(g) \right\}_n$. Then, as a consequence of Theorem \ref{thm:precond-toeplitz-seq-distrib}, we can conclude that the preconditioned symmetrized sequence $\left\{ \abs{P_n}^{-1} Y_n T_n(g) \right\}_n$ is distributed as $\phi_1$ in the sense of eigenvalues, which implies the desired cluster at $\pm 1$.

\begin{lemma}\label{lemma:_uni-level_P}
    With $g$ defined as in \eqref{eq:g-1d-def} and $P_n$ as in \eqref{eq:precond-def-2d}, it holds
    \[ \left\{ P_n^{-1} T_n(g) \right\}_n \sim_\sigma 1. \]
\end{lemma}
{
\begin{proof}
    From GLT Axiom 2, we know that $\left\{T_n(g)\right\}_n\sim_{\rm GLT} g$. To find the GLT symbol of $\left\{P_n\right\}_n$, we focus on the matrices involved in \eqref{eq:precond-def-1d}:
    \begin{itemize}
        \item by GLT 2, the symbol of $\left\{T_n(2-2\cos{\theta})\right\}_n$ is its generating function $2-2\cos\theta$;
        \item the symbol of $\left\{\tau\big(T_n(v_\alpha) + T_n(v_\alpha)^\top \big)\right\}_n$ is the same as $\left\{T_n(v_\alpha) + T_n(v_\alpha)^\top\right\}_n$, which is easily computed as $v_\alpha(\theta) + v_\alpha(-\theta) = v_\alpha(\theta) + \overline{v_\alpha(\theta)}$ due to the properties of Toeplitz matrices and the GLT axioms;
            \item the symbol of $\{I_n\}_n$ is trivially 1.
    \end{itemize}
    We combine these results through GLT Axioms 3 and 4 to obtain the following overall GLT symbol of \cred{$P_n$:}
    \begin{equation*}
        \Big[
        \nu^2 + \nu \big(d_++d_-\big) \big(v_\alpha + \overline{v_\alpha} \big) +  (d_+ - d_-)^2 (2-2\cos{\theta})^\alpha + d_+ d_- ( v_\alpha + \overline{v_\alpha} )^2
        \Big]^\frac{1}{2},
    \end{equation*}
    which is equal to $\abs{g}$. A simple application of GLT 3 leads to $\{P_n^{-1} T_n(g) \}_n \sim_{\rm GLT} \frac{g}{\abs{g}}$ and by GLT 1 we have $\{P_n^{-1} T_n(g) \}_n \sim_\sigma \frac{g}{\abs{g}}$, or equivalently $\{P_n^{-1} T_n(g) \}_n \sim_\sigma \frac{\abs{g}}{\abs{g}} = 1$.
\end{proof}
}

\begin{remark} \label{rem:Pn-tauprec}
        As a byproduct of the proof of Lemma \ref{lemma:_uni-level_P}, we deduce that $\left\{P_n\right\}_n$ and $\left\{T_n(|g|)\right\}_n$ have the same GLT symbol i.e. $|g|$. The same is true for $\left\{\tau\left(T_n(|g|)\right)\right\}_n$, since as already observed it has the same symbol as $\left\{T_n(|g|)\right\}_n$. Therefore 
\[
\left\{P_n - \tau\left(T_n(|g|)\right)\right\}_n\sim_{\rm GLT} 0,
\]
 which means that $P_n = \tau\left(T_n(|g|)\right) + \mathcal{E}_n$ with $\mathcal{E}_n=P_n - \tau\left(T_n(|g|)\right)$ negligible in the singular value sense and in the eigenvalue sense for large $n$, due to the fact that both $P_n$ and $\tau\left(T_n(|g|)\right)$ are real symmetric.
 
However, from the point of view of the efficiency in the preconditioning $P_n$ is closer to the original coefficient matrix than $\tau\left(T_n(|g|)\right)$, since it is constructed by following the same algebraic operations as in the construction of the original coefficient matrix.
\end{remark}

\subsection{Extension to the multilevel case} 

We proceed to study the multilevel setting, building on the foundation set in the one-dimensional case. For simplicity, in \eqref{eq:fde} we take $k=2$, but the results can be easily generalized to any $k>2$. The equations are discretized with the same method used for the one-dimensional case, with space step sizes $h_i:=\frac{b_i-a_i}{n_i+1}$ for $n_i\in\mathbb{N}$, $i=1,2$ and time step size $\tau:=\frac{T}{l}$. It leads to the following linear system
\begin{equation} \label{2dlinearsysm}
    \Big[ I_\mi{n} + I_{n_2} \otimes A_{n_1} + A_{n_2} \otimes I_{n_1} \Big] \vet{u}^{(l)} = \vet{u}^{(l-1)} + \tau \vet{f}^{(l)},
\end{equation}
where $\mi{n}=(n_1,n_2)$ and
\begin{equation*}
    A_{n_i} := \frac{\tau}{h_i^{\alpha_i}} \Big[ d_{+,i} T_{n_i}(v_{\alpha_i}) + d_{-,i} T_{n_i}(v_{\alpha_i})^\top \Big],
    \qquad i=1,2,
\end{equation*}
in which the generating functions $v_{\alpha_i}$ are defined by \eqref{eq:v_alpha-def}. By linearity, for $i=1,2$, $A_{n_i}$ is the Toeplitz matrix of size $n_i$ generated by
\begin{equation*}
    w_i(\theta) := \frac{\tau}{h_i^{\alpha_i}} \big[ d_{+,i} v_{\alpha_i}(\theta)+d_{-,i} v_{\alpha_i}(-\theta) \big] = \frac{\tau}{h_i^{\alpha_i}} \big[ d_{+,i} v_{\alpha_i} + d_{-,i} \overline{v_{\alpha_i}} \big],
\end{equation*}
that is $A_{n_i} = T_{n_i}(w_i)$. Therefore, we conclude that the coefficient matrix has the 2-level Toeplitz structure
\begin{equation*}
    I_\mi{n} + I_{n_2} \otimes T_{n_1}(w_1) + T_{n_2}(w_2)\otimes I_{n_1}
\end{equation*}
in which by Proposition \ref{prop:separable-toeplitz} we have
\begin{align*}
    I_{n_2}\otimes T_{n_1}(w_1) &= T_{n_2}(1) \otimes T_{n_1}(w_1) = T_\mi{n}(1\otimes w_1), \\
    T_{n_2}(w_2)\otimes I_{n_1} &= T_{n_2}(w_2) \otimes T_{n_1}(1) = T_\mi{n}(w_2 \otimes 1).
\end{align*}
By linearity of the Toeplitz operator, the coefficient matrix becomes
\begin{equation*}
    I_\mi{n} + T_\mi{n}\big(w_1(\theta_1)\big) + T_\mi{n}\big(w_2(\theta_2)\big) = T_\mi{n}\big(g(\boldsymbol{\theta})\big)
\end{equation*}
with generating function
\begin{equation} \label{eq:g-2d-def}
    g(\boldsymbol{\theta}) := 1 + w_1(\theta_1) + w_2(\theta_2),
\end{equation}
so that the linear system of interest is rewritten as
\begin{equation*}
    T_\mi{n}(g) \vet{u}^{(l)} = \vet{u}^{(l-1)} + \tau \vet{f}^{(l)}.
\end{equation*}
Premultiplying it by the 2-level anti-identity matrix $Y_\mi{n} = Y_{n_1 n_2}$, we obtain the symmetric coefficient matrix $Y_\mi{n}T_\mi{n}(g)$. Like in the unilevel case, Theorem \ref{thm:symm-mat-distr-multilev} provides the spectral symbol of the related \cred{ matrix sequence:}
\begin{equation*}
    \{ Y_\mi{n}T_\mi{n}(g) \}_\mi{n} \sim_\lambda
    \phi_{\abs{g}} = \left\{
    \begin{array}{cl}
        \abs{g(\boldsymbol{\theta})}, & \boldsymbol{\theta}\in [0,2\pi]^2, \\
        -\abs{g(-\boldsymbol{\theta})}, & \boldsymbol{\theta}\in  [-2\pi,0)^2.
    \end{array} \right.
\end{equation*}
Now, the multilevel version of the preconditioner $P_n$ proposed in the previous section is defined as
\begin{equation} \label{eq:precond-def-2d}
    P_{\mi{n}} := I_\mi{n} + I_{n_2} \otimes R_{n_1} + R_{n_2} \otimes I_{n_1},
\end{equation}
where for $i=1,2$
\begin{equation*}
    R_{n_i} := \frac{\tau}{h_i^{\alpha_i}} \Big[ (d_{+,i} -d_{-,i} )^2 \big(T_{n_i}(2-2\cos{\theta_i}) \big)^{\alpha_i}  + d_{+,i} d_{-,i} \tau\big( T_n(v_{\alpha_i}) + T_n(v_{\alpha_i})^\top \big)^2 \Big]^{\frac{1}{2}}.
\end{equation*}
The construction of $R_{n_i}$ is based on $\abs{w_i} = (w_i\overline{w_i})^\frac{1}{2}$. In fact, with computations similar to \eqref{eq:g-abs-calc},
\begin{align*}
    w_i\overline{w_i}
    &= \Big(\frac{\tau}{h_i^{\alpha_i}}\Big)^2 \Big[ d_{+,i} v_{\alpha_i} + d_{-,i} \overline{v_{\alpha_i}} \Big] \Big[ d_{+,i} \overline{v_{\alpha_i}} + d_{-,i} v_{\alpha_i} \Big] \\
    &= \Big(\frac{\tau}{h_i^{\alpha_i}}\Big)^2 \Big[
    \big( d_{+,i}^2 + d_{-,i}^2 \big) \abs{v_{\alpha_i}}^2
    + d_{+,i}d_{-,i}\big( v_{\alpha_i}^2 + \overline{v_{\alpha_i}}^2 \big) \Big] \\
    &= \Big(\frac{\tau}{h_i^{\alpha_i}}\Big)^2 \Big[
    \big( d_{+,i} - d_{-,i} \big)^2 \abs{v_{\alpha_i}}^2
    + d_{+,i}d_{-,i}\big( v_{\alpha_i} + \overline{v_{\alpha_i}} \big)^2 \Big]
    \end{align*}
and in conclusion
\begin{align*}
    \abs{w_i} = \frac{\tau}{h_i^{\alpha_i}} \Big[
    \big( d_{+,i} - d_{-,i} \big)^2 (2-2\cos{\theta_i})^{\alpha_i}
    + d_{+,i}d_{-,i}\big( v_{\alpha_i} + \overline{v_{\alpha_i}} \big)^2 \Big]^\frac{1}{2}.
\end{align*}

Like in the unilevel case, we prove the effectiveness of $P_\mi{n}$ by applying Theorem \ref{thm:precond-toeplitz-seq-distrib}, whose hypothesis is satisfied due to the following lemma and Remark \ref{remark:acs-cluster-equiv}.

\begin{lemma}
    With $g$ defined as in \eqref{eq:g-2d-def} and $P_\mathbf{n}$ as in \eqref{eq:precond-def-1d}, it holds
    \begin{equation*}
        \big\{ P_\mi{n}^{-1} T_\mi{n}(g) \big \}_\mi{n} \sim_\sigma 1.
    \end{equation*}
\end{lemma}

\begin{proof}
    By GLT Axiom 2, $\left\{T_\mathbf{n}(g)\right\}_\mathbf{n}\sim_{\rm GLT} g$.  Because of Remark \ref{rem:Pn-tauprec}, $P_n = T_n(|g|) + \mathcal{E}_n(g)$ with $\left\{\mathcal{E}_n\right\}_n\sim_{\rm GLT} 0$. Hence, looking at the definition of $P_\mathbf{n}$, we deduce that $P_\mathbf{n} = T_\mathbf{n}(|g|) + \mathcal{E}_\mathbf{n}(g)$ with $\{\mathcal{E}_\mathbf{n}\}_\mathbf{n}\sim_{\rm GLT} 0$. Then, $\{P_\mathbf{n}\}_\mathbf{n}$ and $\{T_\mathbf{n}(g)\}_\mathbf{n}$ are GLT matrix sequences with symbols $|g|$ and $g$, respectively. Therefore, by GLT Axiom 3 we have $\{ P_\mathbf{n}^{-1} T_\mathbf{n}(g) \}_\mathbf{n}\sim_{\rm GLT} \frac{g}{|g|}$ and finally the preconditioned matrix sequence is distributed as 1 in the singular value sense.
\end{proof}

We conclude that the preconditioned symmetrized sequence $\big\{ \abs{P_\mi{n}}^{-1} Y_\mi{n} T_\mi{n}(g) \big \}_\mi{n}$ is spectrally distributed as $\phi_1$, demonstrating the effectiveness of the proposed preconditioner.

\cred{
\begin{remark}
    We remark that the aforementioned preconditioning techniques can be straightforwardly extended to cases with second-order accuracy in both time and space. Specifically, the Crank–Nicolson method and the weighted and shifted Gr\"unwald scheme \cite{Tian2015} are used for \eqref{eq:fde}.
\end{remark}
}

\section{Spectral analysis and preconditioning proposal for ill-conditioned symmetric multilevel Toeplitz systems}\label{sec:symmetric}

In this section, we {shift our focus to a different class of problems:} multilevel {symmetric} Toeplitz matrices characterized by generating functions that possess zeros {of noninteger order} at the origin. {The symmetric nature of the matrices makes the spectral analysis simpler, as it does not require the use of GLT theory, and allows the use of the \cred{PCG} method. \cred{We propose an ad hoc $\tau$ preconditioner} and we are able to prove that the spectrum of the preconditioned coefficient matrix is contained in a positive constant interval, so that the condition number of the matrix is bounded from above by a constant and the PCG method can solve the system in a number of iterations independent from the matrix size}.

{In what follows, we consider  the ill-conditioned symmetric multilevel Toeplitz matrix generated by a nonnegative integrable even function $p_{\boldsymbol{\alpha}}(\boldsymbol{\theta}) \in L^1\left([-\pi, \pi]^k\right)$ such that $p_{\boldsymbol{\alpha}}(\boldsymbol{0})=0$, where $\boldsymbol{\alpha}=\left(\alpha_1, \ldots, \alpha_k\right)$ with each $\alpha_i \in(1,2)$ and $\boldsymbol{\theta}=\left(\theta_1, \ldots, \cred{\theta_k}\right) \in[-\pi, \pi]^k$. The technique we adopt draws from the one used in \cite{huang2021spectral}, where the same kind of function is studied in a similar setting.}

{Using \eqref{multilevel_fenefun}, for every $\mi{n}=(n_1,\ldots,n_k)$ the function} $p_{\boldsymbol{\alpha}}(\boldsymbol{\theta})$ generates a $k$-level symmetric Toeplitz matrix ${B_{\mathbf{n}}} := T_{\mathbf{n}}(p_{\boldsymbol{\alpha}})\in{\mathbb{R}^{N(\mi{n})\times N(\mi{n})}}$, {in which} each block at the $i$-th level is a Toeplitz matrix of size $n_i \times n_i$, for $i = 1, \ldots, k$. We then consider the {following linear system}
\begin{equation}\label{eqn:ill_systm}
{B}_{\mathbf{n}}\mathbf{u}=\mathbf{b},
\end{equation}
where $\mathbf{u} \in {\mathbb{R}^{N(\mi{n})}}$ is the {vector of unknowns} and $\mathbf{b} \in {\mathbb{R}^{N(\mi{n})}}$ is the right-hand side vector.

{Our proposed preconditioner for \eqref{eqn:ill_systm} is the $\tau$ preconditioner $\tau(R_{\mathbf{n}})$ derived from the $k$-level Toeplitz matrix
\begin{equation} \label{eqn:Rn-def}
{R}_{\mathbf{n}} := \sum_{i=1}^k \left[ I_{{n_1}\cdots n_{i-i}} \otimes l_i R_{n_i}^{\left(\alpha_i\right)} \otimes I_{n_{i+1} \cdots n_k} \right],
\end{equation}
in which the coefficients $l_i$ are real constants and $R_{n_i}^{(\alpha_i)}$ is the $n_i$-th Toeplitz matrix generated by
\begin{equation*} 
r_{\alpha_i}\left(\theta_i\right) := \left(2-2 \cos\theta_i\right)^{\frac{\alpha_i}{2}} = \left| 2\sin\left(\frac{\theta_i}{2}\right) \right|^{\alpha_i}.
\end{equation*}
In other words, $R_{\mathbf{n}}$ is the multilevel Toeplitz matrix generated by
\begin{equation*} 
    r_{\boldsymbol{\alpha}}(\boldsymbol{\theta}) := \sum_{i=1}^k l_i r_{\alpha_i} (\theta_i).
\end{equation*}
By the results contained in \cite{Celik_Duman_2012}, $R_{n_i}^{(\alpha_i)}$ has the following expression
\begin{equation*}
R_{n_i}^{(\alpha_i)} := T_{n_i}(r_{\alpha_i}) = \left[
\begin{array}{ccccc}
\rho_0^{(\alpha_i)} & \rho_1^{(\alpha_i)} & \rho_2^{(\alpha_i)} & \ldots & \rho_{n_i-1}^{(\alpha_i)} \\
\rho_1^{(\alpha_i)} & \rho_0^{(\alpha_i)} & \rho_1^{(\alpha_i)} & \cdots & \rho_{n_i-2}^{(\alpha_i)} \\
\rho_2^{(\alpha_i)} & \rho_1^{(\alpha_i)} & \rho_0^{(\alpha_i)} & \cdots & \rho_{n_i-3}^{(\alpha_i)} \\
\vdots & \vdots & \vdots & \ddots & \vdots \\
\rho_{n_i-1}^{(\alpha_i)} & \rho_{n_i-2}^{(\alpha_i)} & \rho_{n_i-3}^{(\alpha_i)} & \cdots & \rho_0^{(\alpha_i)}
\end{array}\right]
\in \mathbb{R}^{n_i \times n_i},
\end{equation*}
with coefficients $\rho_j^{(\alpha_i)}$ defined as
\begin{equation}\label{eqn:coeff_r_alpha}
\rho_j^{(\alpha_i)} := \frac{(-1)^j \Gamma(\alpha_i+1)}{\Gamma\left(\frac{\alpha_i}{2}-j+1\right) \Gamma\left(\frac{\alpha_i}{2}+j+1\right)}, \qquad j=1, \dots, n_i-1.
\end{equation}
The coefficients $\rho_j^{(\alpha_i)}$ satisfy the following properties.}
\begin{lemma}
[\cite{Celik_Duman_2012}] \label{lemma:property_toep_entry}  Let $\rho_j^{(\alpha_i)}$ be defined by \eqref{eqn:coeff_r_alpha} and $1<\alpha_i \leq 2$ for $i=1,\dots, k$. Then,
\begin{equation*}
\rho_0^{(\alpha_i)}=\frac{\Gamma(\alpha_i+1)}{\Gamma\left(\frac{\alpha_i}{2}+1\right)^2} \geq 0,\qquad \rho_{-j}^{(\alpha_i)}=\rho_j^{(\alpha_i)} \leq 0, \qquad j\in\mathbb{N},
\end{equation*}
\begin{equation*}
\sum_{j=-\infty}^{\infty} \rho_j^{(\alpha_i)}=0, \qquad 2 \sum_{j=1}^{\infty} \rho_j^{(\alpha_i)}+\rho_0^{(\alpha_i)}=0 .
\end{equation*}
\end{lemma}

To obtain {bounds on} the spectrum of the preconditioned matrix $\tau({R}_{\mathbf{n}})^{-1} {B}_{\mathbf{n}}$, the following lemmas are crucial.

\begin{lemma}\cred{\cite[Lemma 7]{huang2022preconditioners}}\label{lemma:eig_tauinvtoep}
 Let $T_n=\left[t_{|i-j|}\right]$ be a symmetric Toeplitz matrix and $\tau\left(T_n\right)$ be the corresponding $\tau$ matrix derived from $T_n$. If the entries of $T_n$ are equipped with {one of} the following properties
$$
t_0>0, t_1<t_2<t_3<\cdots<t_{n-1}<0 \quad \text {and} \quad t_0+2 \sum_{i=1}^n t_i>0 \text {, }
$$
or
$$
t_0<0, t_1>t_2>t_3>\cdots>t_{n-1}>0 \quad \text {and} \quad t_0+2 \sum_{i=1}^n t_i<0,
$$
{then} the eigenvalues of {the} matrix $\tau\left(T_n\right)^{-1} T_n$ satisfy
$$
\frac{1}{2}<\lambda\left(\tau\left(T_n\right)^{-1} T_n\right)<\frac{3}{2}.
$$
\end{lemma}
\begin{lemma}\label{lemma:tauinv_R}
    For any nonzero vector $\mathbf{z} \in{\mathbb{R}^{N(\mi{n})}}$, it follows that
$$
\frac{1}{2}<\frac{\mathbf{z}^{\top} {R}_{\mathbf{n}} \mathbf{z}}{\mathbf{z}^{\top} \tau({R}_{\mathbf{n}}) \mathbf{z}}<\frac{3}{2}.
$$
\end{lemma}
\begin{proof}
    {The inequality follows} as a direct consequence of Lemmas \ref{lemma:property_toep_entry} {and} \ref{lemma:eig_tauinvtoep} {combined with the min-max theorem for eigenvalues}.
\end{proof}

{Moreover, we define the auxiliary matrix $Q_{\mathbf{n}}:=T_{\mathbf{n}}(q_{\boldsymbol{\alpha}})$ in which}
\begin{equation}\label{eqn:q_alpha}
q_{\boldsymbol{\alpha}}(\boldsymbol{\theta}) := \sum_{i=1}^k l_i\left|\theta_i\right|^{\alpha_i},
\end{equation}
where $l_i>0$ for $i=1, \ldots, k$ are {real} constants. As in \cite{huang2021spectral}, to facilitate our analysis {we assume that the functions $p_{\boldsymbol{\alpha}}(\boldsymbol{\theta})$ and $q_{\boldsymbol{\alpha}}(\boldsymbol{\theta})$} meet the following requirement.

\begin{hypothesis} \label{assumption}
    There exists two positive constants $c_0<c_1$ such that
$$
0<c_0 \leq \frac{p_{\boldsymbol{\alpha}}(\boldsymbol{\theta})}{q_{\boldsymbol{\alpha}}(\boldsymbol{\theta})} \leq c_1 .
$$
\end{hypothesis}

Note that this assumption implies that {$p_{\boldsymbol{\alpha}}(\boldsymbol{\theta})$ has a zero of order $\alpha_i$ at $\theta_i=0$ on each direction $i$, for $i=1, \ldots, k$.}

{The following lemma stems directly from Proposition \ref{prop:toeplitz-f-frac-g}}.

\begin{lemma} \label{lemma:Qinv_B} Let ${B}_{\mathbf{n}}, {Q}_{\mathbf{n}}$ be the multilevel Toeplitz matrices generated by $p_{\boldsymbol{\alpha}}(\boldsymbol{\theta})$ and $q_{\boldsymbol{\alpha}}(\boldsymbol{\theta})$, respectively. Then, for any nonzero vector $\mathbf{z} \in {\mathbb{R}^{N(\mi{n})}}$, we have\cred{
$$
c_0 {<} \frac{\mathbf{z}^{\top} {B}_{\mathbf{n}}\mathbf{z}}{\mathbf{z}^{\top} {Q}_{\mathbf{n}} \mathbf{z}} {<} c_1.
$$}
\end{lemma}

\begin{lemma}\label{lemma:Rinv_Q}
{Given the multilevel Toeplitz matrix $R_{\mathbf{n}}$ defined in \eqref{eqn:Rn-def}}, for any nonzero vector $\mathbf{z} \in {\mathbb{R}^{N(\mi{n})}}$ we have \[
1<\frac{\mathbf{z}^{\top}  {Q}_{\mathbf{n}}\mathbf{z}}{\mathbf{z}^{\top} {R}_{\mathbf{n}} \mathbf{z}}<\frac{\pi^2}{4}.\]
\end{lemma}
\begin{proof}
It is easy to check that
\begin{equation*}
\min_i  \frac{\left|\theta_i\right|^{\alpha_i}}{r_{\alpha_i}\left(\theta_i\right)} \leq \frac{q_{\boldsymbol{\alpha}}(\boldsymbol{\theta})}{r_{\boldsymbol{\alpha}}(\boldsymbol{\theta})}=\frac{\sum_{i=1}^{{k}} l_i\left|\theta_i\right|^{\alpha_i}}{\sum_{i=1}^{{k}} l_i r_{\alpha_i}\left(\theta_i\right)} \leq \max _i \frac{\left|\theta_i\right|^{\alpha_i}}{r_{\alpha_i}\left(\theta_i\right)} .
\end{equation*}
{Noting that for any $\alpha\in (1,2)$ it holds}
$$
1 <\frac{\theta^\alpha}{\left(2\sin \frac{\theta}{2}\right)^\alpha}= \frac{\left(\frac{\theta}{2}\right)^\alpha}{\left(\sin \frac{\theta}{2}\right)^\alpha} \leq \frac{\pi^2}{4},
$$
\cred{which follows}
$$
1 \leq \frac{q_{\boldsymbol{\alpha}}(\boldsymbol{\theta})}{r_{\boldsymbol{\alpha}}(\boldsymbol{\theta})} \leq \frac{\pi^2}{4} .
$$
{Proposition \ref{prop:toeplitz-f-frac-g}} implies that
$$
1 \leq \min  \frac{q_{\boldsymbol{\alpha}}(\boldsymbol{\theta})}{r_{\boldsymbol{\alpha}}(\boldsymbol{\theta})} {<}  \lambda\left({R}_{\mathbf{n}}^{-1} {Q}_{\mathbf{n}}\right) {<} \max  \frac{q_{\boldsymbol{\alpha}}(\boldsymbol{\theta})}{r_{\boldsymbol{\alpha}}(\boldsymbol{\theta})} \leq \frac{\pi^2}{4}.
$$
Therefore, for any nonzero vector $\mathbf{z} \in {\mathbb{R}^{N(\mi{n})}}$, we have $$
1<\frac{\mathbf{z}^{\top}  {Q}_{\mathbf{n}}\mathbf{z}}{\mathbf{z}^{\top} {R}_{\mathbf{n}} \mathbf{z}}<\frac{\pi^2}{4}.
$$ The proof is complete.
\end{proof}

{At} last, we are ready to prove the effectiveness of $\tau({R}_{\mathbf{n}})$ for ${B}_{\mathbf{n}}$.
\begin{theorem}\label{them:preinv_coefsystm}
{For any $\mi{n}\in\mathbb{N}^k$,} the eigenvalues of {the} preconditioned matrix $\tau({R}_{\mathbf{n}})^{-1} {B}_{\mathbf{n}}$ are bounded {from} below and above by {positive} constants {$c_0,c_1$} independent of the matrix size, i.e.,
$$
\frac{c_0}{2} \leq \lambda\left(\tau({R}_{\mathbf{n}})^{-1} {B}_{\mathbf{n}}\right) \leq \frac{3 \pi^2 c_1 }{8} .
$$
\end{theorem}
\begin{proof}
Combining Lemmas \ref{lemma:tauinv_R}, \ref{lemma:Qinv_B} and \ref{lemma:Rinv_Q}, {we obtain} that
$$
\frac{c_0}{2} \leq \frac{\mathbf{z}^{\top} {B}_{\mathbf{n}}\mathbf{z}}{\mathbf{z}^{\top} \tau({R}_{\mathbf{n}}) \mathbf{z}}=\frac{\mathbf{z}^{\top} {R}_{\mathbf{n}}\mathbf{z}}{\mathbf{z}^{\top} \tau({R}_{\mathbf{n}}) \mathbf{z}} \cdot \frac{\mathbf{z}^{\top} {Q}_{\mathbf{n}}\mathbf{z}}{\mathbf{z}^{\top} {R}_{\mathbf{n}} \mathbf{z}} \cdot \frac{\mathbf{z}^{\top} {B}_{\mathbf{n}} \mathbf{z}}{\mathbf{z}^{\top} {Q}_{\mathbf{n}} \mathbf{z}} \leq \frac{3 \pi^2 c_1 }{8} .
$$
Invoking the {min-max} theorem, we have
$$
\frac{c_0}{2} \leq \lambda\left(\tau({R}_{\mathbf{n}})^{-1} {B}_{\mathbf{n}}\right) \leq \frac{3 \pi^2 c_1 }{2},
$$
{which concludes the proof}.
\end{proof}
This result implies that the spectrum of the preconditioned matrix is uniformly bounded with respect to $\mathbf{n}$ and, crucially, {that} the smallest eigenvalue is bounded away from zero. This, in turn, ensures the linear convergence rate when employing the PCG method to solve the ill-conditioned multilevel Toeplitz systems {\eqref{eqn:ill_systm}}. {Another immediate consequence is the following upper bound on the condition number}.
\begin{corollary}
The (spectral) condition number of the preconditioned matrix $\tau({R}_{\mathbf{n}})^{-1} {B}_{\mathbf{n}}$ is bounded by a constant independent of the matrix size.
\end{corollary}
\begin{proof}
From Theorem \ref{them:preinv_coefsystm}, we obtain
$$
\kappa_2\left(\tau({R}_{\mathbf{n}})^{-1} {B}_{\mathbf{n}}\right)=\frac{\lambda_{\max }\left(\tau({R}_{\mathbf{n}})^{-1} {B}_{\mathbf{n}}\right)}{\lambda_{\min }\left(\tau({R}_{\mathbf{n}})^{-1} {B}_{\mathbf{n}}\right)} \leq \frac{3 \pi^2 c_1 }{c_0}.
$$ The proof is complete.
\end{proof}
{We conclude that} the number of iterations required by the PCG method to solve system \eqref{eqn:ill_systm} using our proposed preconditioner {is} independent of the matrix size.

\section{Numerical examples}\label{sec:numerical}
In this section, we demonstrate the effectiveness of our proposed preconditioner, against the state-of-the-art solver proposed in \cite{Pestana2019,huang2021spectral,Hon_2024}. All numerical experiments are carried out using \cred{ MATLAB R2021a on a Dell R640 server with dual Xeon Gold 6246R 16-Cores 3.4 GHz CPUs and 512GB RAM running Ubuntu $20.04$ LTS.}
The CPU time in seconds is measured using the MATLAB built-in functions \textbf{tic} and \textbf{toc}. Our proposed preconditioner $P_n$ is implemented by the function \textbf{dst} (discrete sine transform). Furthermore, the MINRES solver and the PCG solver are implemented using the built-in functions \textbf{minres} and \textbf{pcg}, respectively. We choose the zero vector as an initial guess unless otherwise stated and a stopping tolerance of $10^{-8}$ based on the reduction in relative residual norms for MINRES and PCG.

In what follows, we will test our preconditioner in the numerical tests conducted in \cred{\cite{Pestana2019,huang2021spectral,Hon_2024}} for comparison purposes. For Examples \cred{\ref{example_1} - \ref{example:two_dim_second}}, the notation $A_n^{R}$ and $A_n^{M}$ are used to denote the existing Toeplitz preconditioners proposed in the same work \cite{Pestana2019}, while $|C_n|:=\sqrt{C_n^{\top} C_n}$ is the absolute value optimal \cite{doi:10.1137/0909051} circulant preconditioner. As in \cite{Pestana2019}, since the direct application of $A_n^{R}$ or $A_n^{M}$ is prohibitively expensive, a multigrid approximation is employed and these multigrid approximations are denoted as MG$(A_n^{R})$ and MG$(A_n^{M})$, respectively. The preconditioner proposed in \cite{Hon_2024} is denoted by $\tau_{n,H}$, which is constructed based on a $\tau$ matrix approximation to the Hermitian part of multilevel Toeplitz matrix. For Example \ref{ex:2levelgeneratingfun}, ‘$S_n$’, ‘$\tau_{n}$’ and ‘$\widehat{\tau}_n$’ represent the PCG method with the Strang \cite{Strang1986} circulant preconditioner, the $\tau$ preconditioner directly derived from the multilevel Toeplitz matrix, and the $\tau$ preconditioner in \cite[Example 4]{huang2021spectral}, respectively. Throughout the numerical tests, the notation ‘MINRES-$Z_n$’ denotes the MINRES solver is implemented with $Z_n$ as the preconditioner.

\begin{example} \cite[Example 2]{Huckle_SC_TP_2005} \label{example_1}
{\rm
In the first example we solve a Toeplitz system with $T_n(f)$ generated by the  function $$f(\theta ) = (2  -  2\cos(\theta ))(1 + \mathbf{i}\theta ).$$ The right-hand side is a random vector constructed using the MATLAB function \textbf{randn} and the initial guess used is $x_0=(1,1,\dots,1)^{\top}/\sqrt{n}$.

Since
\begin{align*}
|f| &= \sqrt{(2-2\cos{(\theta)})^2+(2-2\cos{(\theta)})^2\theta^2} \\
&\approx \sqrt{(2-2\cos{(\theta)})^2+(2-2\cos{(\theta)})^3},
\end{align*}
we construct the following $\tau$ preconditioner whose spectral behavior mimics the symbol $|f|$
\begin{align*}
P_n = \sqrt{T_n[2-2\cos(\theta )]^2 +  T_n[2-2\cos(\theta )]^3 }.
\end{align*}

Since $\Tilde{g} = \sqrt{(2-2\cos{(\theta)})^2+(2-2\cos{(\theta)})^3}$ is a good approximation of $\abs{f}$ having the same zero at $\theta=0$ of the same order, the singular values of $\{P_n^{-1} T_n(\Tilde{g})\}_n$ will cluster at a small positive interval $[\check{a},\hat{a}]$. Hence, the eigenvalues of $\{P_n^{-1} Y_n T_n(\Tilde{g})\}_n$ will cluster at $\pm [\check{a},\hat{a}]$.

Table \ref{Table:example1} shows the iteration numbers and CPU times of MINRES with the preconditioners compared. The table shows that (i) MINRES-$P_n$ is the most efficient among the compared preconditioners in terms of computational time; (ii) despite having the fewest number of iterations, MINRES-\cred{MG$(A_n^{M})$} incurs the highest CPU time consumption when solving the linear system; (iii) with the exception of preconditioner $|C_n|$, the iteration number of the remaining preconditioners remains stable as the spatial grid gets refined. The bounded iteration number of preconditioners means that its convergence rate does not deteriorate as the grids getting refined, which supports the theoretical results and demonstrates the robustness of the proposed solver.
}

\begin{table}[h!]
\caption{Iteration numbers and CPU times of MINRES for Example \ref{example_1}}
\label{Table:example1}
\begin{centering}
\begin{tabular}{c|cc|cc|cc|cc}
\hline
 \multirow{2}{*}{$n$} & \multicolumn{2}{c|}{$|C_n|$} & \multicolumn{2}{c|}{$A_n^{R}$} & \multicolumn{2}{c|}{MG$(A_n^{M})$} & \multicolumn{2}{c}{$P_n$} \\  
 \cline{2-9}
& Iter         & CPU(s)        & Iter        & CPU(s)            & Iter          & CPU(s)         & Iter       & CPU(s)       \\ \hline
4095                 & 162          & 0.26          & 62          & 0.066              & 24            & 0.12           & 26         & 0.032        \\
8191                 & 316          & 1.2           & 62          & 0.11              & 25            & 0.34           & 27         & 0.088         \\
16383                 & 524          & 3.2           & 62          & 0.19               & 25            & 0.62           & 26         & 0.14         \\
32767                 & 856          & 10           & 64          & 0.40            & 25            & 1.3           & 27         & 0.33         \\ \hline
\end{tabular}
\end{centering}
\end{table}

\end{example}

\begin{example}\cite[Example 3]{Pestana2019} \label{ex:2dfpde}
{\rm
We proceed to solve a two-level Toeplitz problem {arising from} the fractional diffusion problem stated in equation \eqref{eq:fde}. {We set
\begin{gather*}
    k=2, \; \Omega=(0,1)\times(0,1), \; T=1, \\
    d_{1,+}=50, \; d_{1,-}=10, \; d_{2,+}=20, \; d_{2,-}=30, \\
    f(x,t)=100 \sin (10 x) \cos (y)+\sin (10 t) x y.
\end{gather*}}

We discretize by the shifted Grünwald-Letnikov method in space, and the backward Euler method in time, which leads to the linear system \eqref{2dlinearsysm}. The number of spatial degrees of freedom in the $x$ and $y$ directions are $n_1$ and $n_2$, respectively; we choose $n_1=n_2=n$. Also,
$h_1=1 /\left(n_1+1\right)$ and $h_2=1 /\left(n_2+1\right)$ are the mesh widths in the $x$ and $y$ directions. Unless ${\alpha_1}={\alpha_2}$, both $\tau / h_1^{\alpha_1}$ and $\tau / h_2^{\alpha_2}$ cannot be independent of $n$; we choose $\tau=1 /\left\lceil n_1^{\alpha_1}\right\rceil$. Stated CPU times and iteration counts are again for the first time step.

The findings presented in Table \ref{Table_Ex2_fractional} demonstrate the following observations: (i) \cred{MINRES-$P_n$ yields iteration counts that are independent of mesh size across a wide range of $({\alpha_1}, {\alpha_2})$; (ii) the behavior of MINRES-$|C_n|$ varies} depending on $({\alpha_1}, {\alpha_2})$. \cred{The difficulty of the problems increases as $n$ grows, resulting in growing iteration counts for MINRES-$|C_n|$}; (iii) our proposed $P_n$ preconditioner demonstrates clear superiority over the compared numerical approaches in terms of both CPU times and robustness with respect to $({\alpha_1}, {\alpha_2})$. This advantage is particularly evident as the values of $({\alpha_1}, {\alpha_2})$ approach $(1,1)$, where the numerical performance of both $\tau_{n,H}$ and MG$(A_n^{R})$ is noticeably poor.
}
\begin{table}
\small \caption{Iteration numbers and CPU times of MINRES for Example \ref{ex:2dfpde} with $d_{1,+}=50, d_{1,-}=10, d_{2,+}=20$, and $d_{2,-}=30$}
\label{Table_Ex2_fractional}
\begin{centering}
\begin{tabular}{cc|cc|cc|cc|cc}
\hline
\multirow{2}{*}{$(\alpha_1,\alpha_2)$} & \multirow{2}{*}{$n^2$} & \multicolumn{2}{c|}{$|C_n|$} & \multicolumn{2}{c|}{MG$(A_n^{R})$} & \multicolumn{2}{c|}{$\tau_{n,H}$} & \multicolumn{2}{c}{$P_n$}\\ \cline{3-10}
                                  &                        & Iter         & CPU(s)        & Iter          & CPU(s)         & Iter       & CPU(s)
                                  & Iter       & CPU(s)\\ \hline
\multirow{4}{*}{(1.01,1.01)}       & 16129                  &99          & 0.69           & 68            & 0.58             & 66 &0.39   & 42    &0.26    \\
                                  & 261121                 & 121         & 4.4           & 68            & 3.8          &68 & 2.58 &42          &1.62          \\
                                  & 4190209                & 134           & 96           & 64            & 1.3e+02     & 64 &49.48 &38    &29.45          \\
                                  & 67092481               & 142           & 1.5e+03        &  62           & 1.8e+03         & 62       &810.66&36  & 482.78       \\ \hline
\multirow{4}{*}{(1.1,1.1)}       & 16129                  & 92           & 0.59          & 40            & 0.32           & 42    &0.25     &30 &0.15    \\
                                  & 261121                 &114           &4.1           & 38            & 2.2              & 42         &1.62  &28&1.14         \\
                                  & 4190209                &126           & 89            &36           & 75             &40      &31.27  &26&20.87        \\
                                  & 67092481               & 130           & 1.4e+03       &  36           & 1.1e+03         &  38    & 504.37  &26&354.37       \\ \hline
\multirow{4}{*}{(1.1,1.5)}       & 16129                  & 141           & 0.88          & 46            & 0.4           & 40    &0.24     &26&0.14    \\
                                  & 261121                 &276           &9.9           & 68            & 3.8              & 46         &1.74  &28&1.10         \\
                                  & 4190209                &453           & 3.2e+02            &90            &  1.7e+02            &46      &36.24  &28 &22.53        \\
                                  & 67092481               & 627           & 6.6e+03       &  114           & 3.3e+03         &  46    & 612.40  &28&386.14      \\ \hline
                                  \multirow{4}{*}{(1.1,1.9)}       & 16129                  & 270           & 1.7          &157            &1.3           &     32      & 0.20 &22&0.13       \\
                                  & 261121                 & 696           & 25           & 432            &24             &   40      &1.52  &26&0.99      \\
                                  & 4190209                &$>1000$           & 7e+02            & $>1000$             & 1.9e+03              &  42       &33.67 &30&24.91       \\
                                  & 67092481               & $>1000$           & 1.1e+04        &   $>1000$           &2.9e+04          &    42     &567.60 & 30& 413.87     \\ \hline
                                  \multirow{4}{*}{(1.5,1.1)}       & 16129                  & 125           & 0.81          & 37            & 0.33           & 18         &0.10   & 17 &0.11       \\
                                  & 261121                 &153           & 5.5           & 51            & 2.9              &18          &0.87 &17 & 0.84          \\
                                  & 4190209                & 133           & 95            & 69            & 1.5e+02             &16     &14.55 & 16 &14.75          \\
                                  & 67092481               & 106           &1.1e+03         & 90            &2.6e+03        &16     &228.94      &15 & 214.36    \\ \hline
                                   \multirow{4}{*}{(1.5,1.5)}       & 16129                  & 102           & 0.65          & 15            & 0.11           & 19         &0.10   & 18 &0.11       \\
                                  & 261121                 &141           & 5           & 14            & 0.85              &18          &0.95 &18 & 1.00          \\
                                  & 4190209                & 153           & 1.1e+02            & 14            & 34             &18     &15.48 & 17 &15.09          \\
                                  & 67092481               & 147           &1.6e+03         & 14            &4.5e+02        &18     &255.43      &16 & 226.24    \\ \hline
                                   \multirow{4}{*}{(1.5,1.9)}       & 16129                  & 183           & 1.2          & 29            & 0.23           & 18         &0.13   & 17 &0.11       \\
                                  & 261121                 &413           & 15           & 43            & 2.4              &18          &0.73 &17 & 0.70          \\
                                  & 4190209                & 733           & 5.2e+02            & 65            & 1.2e+02             &17     &14.74 & 17 &14.76          \\
                                  & 67092481               & $>1000$           &1.1e+04         & 100            &2.9e+03        &17     &253.57      &17 & 254.89    \\ \hline
                                   \multirow{4}{*}{(1.9,1.1)}       & 16129                  & 125           & 0.79          & 131            & 1.1           & 10         &0.060   & 11 &0.097       \\
                                  & 261121                 &85           & 3.1           & 197            & 11              &10          &0.48 &10 & 0.48          \\
                                  & 4190209                & 56           & 41            & 230            & 4.4e+02             &9     &8.23 & 9 &8.89          \\
                                  & 67092481               & 37           &4e+02         & 236            &6.8e+03        &9     &135.31      &9 & 134.94    \\ \hline
                                   \multirow{4}{*}{(1.9,1.5)}       & 16129                  & 155           & 0.97          & 33            & 0.27           & 12         &0.093   & 12 &0.088       \\
                                  & 261121                 &210           & 7.6           & 48            & 2.7              &12          &0.52 &12 & 0.52          \\
                                  & 4190209                & 194           & 1.4e+02            & 69            & 1.6e+02             &11     &9.75 & 11 &10.14          \\
                                  & 67092481               & 161           &1.7e+03         & 97            &2.8e+03        &11     &161.84      &11 & 164.46   \\ \hline
                                   \multirow{4}{*}{(1.9,1.9)}       & 16129                  & 132           & 0.83          & 13            & 0.098           & 11         &0.082   & 11 &0.067       \\
                                  & 261121                 &195           & 7.1           & 13            & 0.8              &11          &0.47 &11 & 0.48          \\
                                  & 4190209                & 215           & 1.5e+02            & 13            & 27             &11     &9.85 & 11 &10.02          \\
                                  & 67092481               & 209           &2.2e+03         & 13            &4.2e+02        &11     &161.24      &11 & 165.21    \\ \hline
\end{tabular}
\end{centering}
\end{table}

\end{example}

\begin{example}\cite[Example 2]{Hon_2024}\label{example:two_dim_second}
{\rm
Here, we consider to solve a two-level Toeplitz problem arising from the two-dimensional fractional diffusion equation \eqref{eq:fde} with
$$
\begin{aligned}
d_{1,+} &=2, \quad d_{1,-}=35, \quad d_{2,+} = 1,  \quad d_{2,-} =20, \quad \Omega=(0,2) \times(0,2), \quad T=1,\\
f\left(x_1, x_2\right)&=  e^t x_1^2(2-x_1)^2 x_2^2(2-x_2)^2 \\
&-e^t x_2^2\left(2-x_2\right)^2 \sum_{i=2}^4 \frac{\binom{2}{i-2} 2^{4-i} i!\left[d_{1,+} x_1^{i-\alpha_1}+d_{1,-}\left(2-x_1\right)^{i-\alpha_1}\right]}{\Gamma\left(i+1-\alpha_1\right)(-1)^{i-2}} \\
& -e^t x_1^2\left(2-x_1\right)^2 \sum_{i=2}^4 \frac{\binom{2}{i} 2^{4-i} i!\left[d_{2,+} x_2^{i-\alpha_2}+d_{2,-}\left(2-x_2\right)^{i-\alpha_2}\right]}{\Gamma\left(i+1-\alpha_2\right)(-1)^{i-2}}.
\end{aligned}\\
$$
The exact solution to the problem under consideration is given by \( u(x_1, x_2, t) = e^t x_1^2(2-x_1)^2 x_2^2(2-x_2)^2 \). For this example, we adopted the weighted and shifted Grünwald scheme and set \( n_1 = n_2 \) and \( \tau = T / (n_1 + 1) \). We refrained from implementing MINRES-\(|C_n|\) and MINRES-MG\((A_n^R)\) due to their prohibitive costs, as evidenced in a prior example. Given that the exact solution is available, we defined the error as \( \text{Err} = \| \mathbf{u} - \widetilde{\mathbf{u}} \|_{\infty} \), where \( \mathbf{u} \) and \( \widetilde{\mathbf{u}} \) represent the exact and approximate solutions, respectively.

The results presented in Table \ref{Table_Ex3_fractional_2nd} elucidate several key observations: (i) MINRES-\(P_n\) exhibits a mesh-independent iteration count and demonstrates robustness across varying parameters; (ii) MINRES-\(P_n\) outperforms MINRES-\(\tau_{n,H}\) notably when the parameters \((\alpha_1, \alpha_2)\) are close to \((1, 1)\); (iii) as the parameters \(\alpha_1\) and \(\alpha_2\) approach towards a value of 2, both MINRES-\(P_n\) and MINRES-\(\tau_{n,H}\) exhibit comparable levels of numerical efficacy.
}
\end{example}

\begin{table}
\small \caption{Iteration numbers and CPU times of MINRES for Example \ref{example:two_dim_second} with $d_{1,+}=2, d_{1,-}=35, d_{2,+}=1$, and $d_{2,-}=20$}
\label{Table_Ex3_fractional_2nd}
\begin{centering}
\begin{tabular}{cc|ccc|ccc}
\hline
\multirow{2}{*}{$(\alpha_1,\alpha_2)$} & \multirow{2}{*}{$n^2$} &  \multicolumn{3}{c|}{$\tau_{n,H}$} & \multicolumn{3}{c}{$P_n$}\\ \cline{3-8}
                                  &     &  Iter          & CPU(s)     &Err     & Iter       & CPU(s)  &Err
                                  \\ \hline
\multirow{4}{*}{(1.01,1.01)}       & 16129                & 174             & 0.72 &4.2e-04   & 76    &0.33 &4.2e-04   \\                     
                                  &261121                   & 141     & 4.05 &9.1e-06 &33    &1.01   &9.1e-06       \\
                                  &  4190209                  & 105         & 95.74       &2.3e-07 &21  & 20.13    &2.3e-07   \\
                                  &  67092481                  & 75         & 942.42       &7.4e-08 &19  & 257.46    &7.6e-08   \\ \hline
\multirow{4}{*}{(1.1,1.1)}       & 16129                  & 74           & 0.36    &5.8e-04     &54 &0.34  &5.8e-04  \\
                                  &261121                & 63             &1.84      &2.4e-05  &32 &1.07    &2.4e-05    \\
                                  &  4190209                  & 53         & 49.09       &1.2e-06 &19  & 19.50    &1.2e-06   \\
                                  &  67092481               & 45         &  574.64    & 6.8e-08  &19 &260.01     & 7.5e-08  \\ \hline
\multirow{4}{*}{(1.1,1.5)}       & 16129                 & 73           & 0.33    &5.6e-04     &36 &0.18 &5.6e-04   \\
                                  &261121                   &  69            &2.04      &3.5e-05  &28 &0.92   &3.5e-05      \\
                                  &  4190209                  & 61         & 56.31       &2.2e-06 &23  & 22.07    &2.2e-06   \\
                                  & 67092481                & 57         &  721.02    & 1.4e-07  &19 &257.07    & 1.4e-07  \\ \hline 
                                  \multirow{4}{*}{(1.1,1.9)}       & 16129                 &66           &    0.38      & 5.1e-04 &32&0.15 & 5.1e-04      \\
                                  & 261121                 & 63             &  1.87       &2.3e-05 &27&0.90   &2.3e-05    \\
                                  &  4190209                  & 61         & 55.96       &1.2e-06 &23  & 21.84    &1.2e-06   \\
                                  & 67092481                &59          &    745.84     &7.5e-08 & 19 & 258.02  &7.0e-08   \\ \hline
                                  \multirow{4}{*}{(1.5,1.1)}       & 16129                  & 59           & 0.28         &6.0e-04   & 33 &0.14  &6.0e-04     \\
                                  & 261121                & 59             &1.78     &3.7e-05 & 27 &0.90   &3.7e-05       \\
                                  &  4190209                  & 53         & 49.13       &2.3e-06 &22  & 21.10    &2.3e-06   \\
                                  &67092481                &51        &647.83     &1.5e-07      &19 & 257.76 &1.5e-07    \\ \hline
                                   \multirow{4}{*}{(1.5,1.5)}       & 16129                   & 25           & 0.13         &5.9e-04   & 24 &0.12 &5.9e-04      \\
                                  & 261121                  & 24             &0.78     &3.7e-05 & 21 &0.75  &3.7e-05        \\
                                  &  4190209                 & 21         & 20.72       &2.3e-06 &18  & 17.56    &2.3e-06   \\
                                  &67092481                    &19     &259.30      &1.5e-07 &17 & 233.45  &1.5e-07  \\ \hline
                                   \multirow{4}{*}{(1.5,1.9)}       & 16129                 & 23           & 0.12         &5.8e-04   & 21 &0.096  &5.8e-04      \\
                                  & 261121            & 23            &0.70     &3.6e-05 & 19 &0.64  &3.6e-05        \\
                                  &  4190209                  & 21         & 20.34       &2.3e-06 &17  & 16.60    &2.3e-06   \\
                                  & 67092481            &19        &258.81     &1.5e-07      &17 & 233.63 &1.5e-07     \\ \hline
                                   \multirow{4}{*}{(1.9,1.1)}       & 16129                  & 51           & 0.24         &3.4e-04   & 27 &0.12  &3.4e-04      \\
                                  & 261121               & 53             &1.58     &1.7e-05 & 25 &0.85   &1.7e-05        \\
                                  &  4190209                  & 53         & 28.98       &1.0e-06 &22  & 21.58    &1.0e-06   \\
                                  & 67092481               &49        &624.69     &6.2e-08      &21  & 281.89 &5.9e-08   \\ \hline
                                   \multirow{4}{*}{(1.9,1.5)}       & 16129                 & 23           & 0.12         &5.2e-04   & 21 &0.11  &5.2e-04     \\
                                  & 261121                  & 23             &0.74     &3.4e-05 & 19 &0.73 &3.4e-05         \\
                                  &  4190209                 & 21         & 20.53       &2.3e-06 &17  & 16.80    &2.3e-06   \\
                                  & 67092481               &19        &259.01     &1.4e-07      &17 & 232.19 &1.4e-07  \\ \hline
                                   \multirow{4}{*}{(1.9,1.9)}       & 16129              & 13           & 0.070         &1.5e-04   & 13 &0.061  &1.5e-04     \\
                                  & 261121          & 12             &0.45     &9.0e-06 & 13 &0.51  &9.0e-06         \\
                                  &  4190209                  & 11         & 11.53       &6.1e-07 &13  & 13.23    &6.1e-07   \\
                                  & 67092481                   &11     &161.52    &4.1e-08   &12 & 172.15 &4.1e-08   \\ \hline
\end{tabular}
\end{centering}
\end{table}

\begin{example}\cite[Example 4]{huang2021spectral}\label{ex:2levelgeneratingfun}
{\rm
Consider a two-level Toeplitz matrix whose generating function is defined by
\begin{equation*}
p_{\boldsymbol{\alpha}}(\boldsymbol{\theta})=p_{\alpha_1}\left(\theta_1\right)+p_{\alpha_2}\left(\theta_2\right)-p_1\left(\theta_1\right) p_1\left(\theta_2\right),
\end{equation*}
where
\begin{equation*}
p_{\alpha_i}\left(\theta_i\right)=\left\{\begin{array}{cl}
\left|\theta_i\right|^{\alpha_i}, & \left|\theta_i\right|<\frac{\pi}{2}, \\
1, & \left|\theta_i\right| \geq \frac{\pi}{2}.
\end{array}\right.
\end{equation*}
It is clear that Hypothesis \ref{assumption} is satisfied, as follows:
\begin{equation*}
0<\frac{4-\pi}{4} \leq \frac{\cred{p_{\boldsymbol{\alpha}}}(\boldsymbol{\theta})}{q_{\boldsymbol{\alpha}}(\boldsymbol{\theta})} \leq 1,
\end{equation*}
where $q_{\boldsymbol{\alpha}}(\boldsymbol{\theta})$ is defined in \eqref{eqn:q_alpha} with $l_1=l_2=1$.

Based on the results presented in Table \ref{table:ex3}, it can be observed that the number of iterations required by the circulant preconditioner $S_n$ and the natural $\tau$ preconditioner $\tau_{n}$ increases significantly as the size of the matrices increases, particularly when $\alpha$ approaches 2. Similarly, the proposed preconditioner $\widehat{\tau}_n$ in \cite{huang2021spectral} also exhibits a rapid increase in iteration numbers as $\alpha$ tends to 1. The iteration numbers and CPU time of the aforementioned preconditioners are clearly dependent on the choice of $(\alpha_1, \alpha_2)$. In contrast, the preconditioner $P_n$ proposed in our study remains stable regardless of the selection of $(\alpha_1, \alpha_2)$. This implies robustness with respect to parameters: both the number of iterations and CPU time exhibit relative stability, thereby addressing the limitations of the aforementioned methods.
}

\begin{table}
\small \caption{Iteration numbers and CPU times of PCG methods for Example \ref{ex:2levelgeneratingfun} with different $(\alpha_1, \alpha_2)$}
\label{table:ex3}
\begin{centering}
\begin{tabular}{cc|cc|cc|cc|cc}
\hline
\multirow{2}{*}{$(\alpha_1,\alpha_2)$} & \multirow{2}{*}{$n^2$} & \multicolumn{2}{c|}{$S_n$} & \multicolumn{2}{c|}{$\tau_{n}$} & \multicolumn{2}{c|}{$\widehat{\tau}_n$} & \multicolumn{2}{c}{$P_n$}\\ \cline{3-10}
&                        & Iter         & CPU(s)        & Iter          & CPU(s)         & Iter       & CPU(s)
& Iter       & CPU(s)\\ \hline
\multirow{4}{*}{(1.01,1.01)}       &16129               & 16         & 0.10           & 9            & 0.06           &79           & 0.50     & 18 & 0.11  \\
& 261121                & 20         & 0.76           & 10            & 0.50           & 113           & 4.85     & 19 & 0.88   \\
& 4190209               & 23         & 13.36            & 10            & 10.41            & 134           & 121.72  & 19 & 17.84      \\
& 67092481                &   25   &   251.95    &    10     &  182.04    &     144       &    2.2e+03 & 23 &     373.67    \\ \hline
\multirow{4}{*}{(1.1,1.1)}       & 16129                  & 18           & 0.12          & 9            & 0.068           & 43    &0.28     & 19&  0.13  \\
& 261121          & 22         &0.85                  & 9            & 0.51
&46           &2.22
&19 &  0.93      \\
& 4190209         &26      &18.63                   &10           & 12.08
&47           & 49.58  & 19 & 19.05       \\
& 67092481  &  30    & 292.43                   &  10           & 177.91   & 48           & 745.38         &24 &  391.95     \\ \hline
\multirow{4}{*}{(1.1,1.5)}       & 16129                  & 23    &0.15          & 10            & 0.067     & 41           & 0.26           &21 & 0.13   \\
& 261121            & 28         &1.15                 & 10            & 0.55              &47           &2.29 &21 & 0.99       \\
& 4190209           &37      &26.10                 &11            &  13.35             &50           & 52.97 & 21&20.34       \\
& 67092481         &  60    & 564.30            &  10           & 178.69   & 50           & 780.07         & 33& 529.84     \\ \hline
\multirow{4}{*}{(1.1,1.9)}       & 16129            &     25      & 0.16                &11            &0.076           & 46           & 0.33 &24 & 0.16      \\
& 261121            &   39      &1.39                & 11            &0.55             & 53           & 2.36  &24 & 1.09     \\
& 4190209           &  59       &40.18                 & 11            & 13.81            &55           & 59.53   &24 & 23.06      \\
& 67092481        &    110     &1.02e+03               &   11          &194.15          & 55           & 851.39 &35 &560.19      \\ \hline
\multirow{4}{*}{(1.5,1.1)}       & 16129            & 23         & 0.16               & 10            & 0.094            & 41           & 0.34  &21  & 0.14      \\
& 261121            &28          &1.33                & 10            & 0.74              &47           & 2.93 &21 &  1.04         \\
& 4190209           &38     &27.05                 & 11            & 13.35           & 49           & 53.87   & 21 &  20.95        \\
& 67092481         &59     &554.81                & 10            &177.92        & 50           &779.51     &  33 & 531.36    \\ \hline
\multirow{4}{*}{(1.5,1.5)}       & 16129           & 24         &0.13                & 11            & 0.059           & 24           & 0.14    & 23 &0.14       \\
& 261121           &36          &1.31               & 12            & 0.56              &25           & 1.12   &23 & 1.04         \\
& 4190209             &63     &43.23                & 15            & 17.44          & 25           & 27.18   & 23 &  22.94       \\
& 67092481           & 132    & 1.2e+03              & 24            &395.59     & 25           &401.93       & 29 &468.35     \\ \hline
\multirow{4}{*}{(1.5,1.9)}       & 16129                  &37            &  0.22         & 13           & 0.10           & 26        &0.18   & 25 & 0.15     \\
& 261121                 &68           & 2.51           & 19            & 0.95              &27          &1.28 & 26&  1.16        \\
& 4190209                & 244           & 162.53            & 42            & 40.86             &27     &25.85 & 26 & 24.89        \\
& 67092481               & 980           &9.1e+03         & 152            &2.3e+03        &27     &430.45      & 38 &  606.80    \\ \hline
\multirow{4}{*}{(1.9,1.1)}       & 16129                  & 25           &0.18           & 11            & 0.09           & 46         &0.58   & 24 &   0.15     \\
& 261121                 & 39          &  1.48          & 11            & 0.59              & 53         &2.48 &24 & 1.07        \\
& 4190209                & 59           & 39.95            & 11            & 12.95             &55     &57.04 & 24 &23.32          \\
& 67092481               & 103           & 955.10        & 11            &193.61        &56     &864.51       & 35& 559.69    \\ \hline
\multirow{4}{*}{(1.9,1.5)}       & 16129             & 36         &0.23                & 13            & 0.08  & 26           & 0.17           &25  &0.18       \\
& 261121            &69          &2.47               & 19            & 0.95             &27           & 1.19   &26 &  1.25         \\
& 4190209            &265     &178.25                & 42            & 47.29            & 27           & 30.23  & 26 & 25.28         \\
& 67092481         & 927    &        9.1e+03       & 126            &1.9e+03        & 27           &434.36      &38 &607.72    \\ \hline
\multirow{4}{*}{(1.9,1.9)}        & 16129                 &   57         & 0.35        & 21            & 0.14           &27           & 0.15       &  27&  0.19     \\
& 261121                 & 208         & 7.07          & 45            & 2.05              &27           & 1.27  & 27 &   1.24     \\
& 4190209                &  997         &  671.16         & 207            & 193.18             &27           & 25.89  &27 & 26.72      \\
& 67092481               & 953           &   9.1e+03      & 994           &1.5e+04        &27     &440.08      &31 & 500.66    \\ \hline
\end{tabular}
\end{centering}
\end{table}

\end{example}

\section{Conclusions}
We introduced a novel preconditioning approach for multilevel Toeplitz systems, utilizing a Krylov subspace method enhanced by multilevel $\tau$-preconditioners. This method is general, as the proposed preconditioner was constructed based solely on the generating function of the multilevel Toeplitz matrix. Specifically, for the nonsymmetric case, we employed a symmetrization strategy, transforming the nonsymmetric system into a symmetric multilevel Hankel system. Subsequently, we performed an asymptotic spectral analysis, examining the spectrum to tailor our preconditioning strategies effectively. For the symmetric case, we rigorously demonstrated that optimal convergence could be achieved using the proposed preconditioner for a class of ill-conditioned multilevel Toeplitz systems. Our numerical results indicated that the preconditioned iterative methods achieved convergence independent of the mesh size and robust across a broad range of problem parameters. Numerical examples showed that our preconditioning strategy consistently outperformed all compared state-of-the-art preconditioned solvers.

As research to be developed in detail in the future, we can mention the use of Theorem \ref{thm:symm-mat-distr-multilev-block}, Theorem \ref{thm:precond-seq-distrib-multilev-block}, and Remark \ref{rem:omega-circ vs tau} in the context of approximation of fractional equations by using \cred{discontinuous Galerkin, finite elements of high order, isogeometric analysis} with intermediate regularity. In all these cases the resulting approximation matrix sequences have matrix-valued symbols (see \cite{GaroniCapizzano_three,GaroniCapizzano_four,systemPDE} and references therein) so that the tools developed here can be conveniently employed. 

\section*{Acknowledgments}
The work of Sean Y. Hon was supported in part by the Hong Kong RGC under grant 22300921 and a start-up grant from the Croucher Foundation. The research of Stefano Serra-Capizzano is funded from the European High-Performance Computing Joint Undertaking  (JU) under grant agreement No 955701. The JU receives support from the European Union’s Horizon 2020 research and innovation programme and Belgium, France, Germany, Switzerland. Furthermore Stefano Serra-Capizzano is grateful for the support of the Laboratory of Theory, Economics and Systems – Department of Computer Science at Athens University of Economics and Business. Finally Stefano Serra-Capizzano and Rosita Sormani are partly supported by Italian National Agency INdAM-GNCS.

\section*{Declarations}
The authors declare that they have no conflict of interest.

\bibliographystyle{plain}

\end{document}